\documentclass[12pt]{amsart}

\usepackage{mathrsfs}
\usepackage{amsfonts}
\usepackage{amssymb}
\usepackage{amsfonts, amscd, amsmath, mathrsfs, amssymb, amsthm, amsxtra, bbding, epsfig, graphicx, latexsym, url, mathbbol, bbold}
\usepackage[papersize={8.3in,11.8in},textwidth=16.8cm,textheight=23.2cm,centering]{geometry}
\usepackage{enumerate}
\usepackage{caption} 

\usepackage{graphicx}
\usepackage{tikz}
\usepackage{float}

\usepackage{xcolor}
\definecolor{cite}{rgb}{0.00,0.60,1.00}
\definecolor{url}{rgb}{1.00,0.10,0.80}
\definecolor{link}{rgb}{0.00,0.00,1.00}
\usepackage[colorlinks,linkcolor=link,urlcolor=url,citecolor=cite,pagebackref,breaklinks]{hyperref}

\hypersetup{
pdfstartpage=1,
pdfstartview=FitH}

\DeclareFontFamily{U}{mathx}{\hyphenchar\font45}
\DeclareFontShape{U}{mathx}{m}{n}{
      <5> <6> <7> <8> <9> <10>
      <10.95> <12> <14.4> <17.28> <20.74> <24.88>
      mathx10
      }{}
\DeclareSymbolFont{mathx}{U}{mathx}{m}{n}
\DeclareMathAccent{\widecheck}{\mathalpha}{mathx}{"71}

\numberwithin{equation}{section}

\allowdisplaybreaks

\newtheorem{theorem}{Theorem}[section]
\newtheorem{lemma}[theorem]{Lemma}

\newtheorem{proposition}[theorem]{Proposition}

\newtheorem{corollary}[theorem]{Corollary}

\newtheorem{remark}[theorem]{Remark}

\makeatletter
\newcounter{roem}
\renewcommand{\theroem}{\Roman{roem}}

\newcommand{\c@org@eq}{}
\let\c@org@eq\c@equation
\newcommand{\org@theeq}{}
\let\org@theeq\theequation

\newcommand{\setroem}{
\let\c@equation\c@roem
 \let\theequation\theroem}

\newcommand{\setarab}{
\let\c@equation\c@org@eq
\let\theequation\org@theeq}
\makeatother

\newtheorem*{claim*}{Claim}

\newcommand{\GL}{\mathrm{GL}}

\usepackage{scalerel}

\DeclareMathOperator{\Mod}{mod}

\renewcommand{\bmod}[1]{\,(\Mod{ #1})}

\newcommand{\R}{\mathbb{R}}
\newcommand{\E}{\mathbb{E}}

\newcommand{\N}{\mathbb{N}}

\newcommand{\U}{\mathbb{U}}
\newcommand{\W}{\mathbb{W}}
\newcommand{\X}{\mathbb{X}}

\newcommand{\Z}{\mathbb{Z}}

\newcommand{\sumh}{\sideset{}{^h}\sum}

\def\le{\leqslant}
\def\leq{\leqslant}

\def\geq{\geqslant}

\usepackage{graphicx}
\usepackage{tikz}

\begin{document}

\title{The typical size of Hecke eigenvalue sums is $o(\sqrt{x})$}

\author{Max Wenqiang Xu}
\address{Yau Mathematical Sciences Center, Tsinghua University, Beijing, 100084, China}
\address{Beijing Institute of Mathematical Sciences and Applications, Beijing, 101408, China}
\email{maxxu1729@gmail.com, maxxu@tsinghua.edu.cn}

\author{Junren Zheng}
\address{School of Mathematics and Statistics, Xi'an Jiaotong University, Xi'an 710049, P. R. China}

\address{D{\'e}partment  de Math{\'e}matiques et Statistique,
Universit{\'e} de Montr{\'e}al, CP 6128 succ Centre-Ville, Montr{\'e}al, QC  H3C 3J7, Canada}
\email{junrenzheng03@gmail.com, junren@stu.xjtu.edu.cn}



\begin{abstract} 
Let $\mathcal{H}_k$ denote the set of normalized
 holomorphic Hecke cusp forms of weight $k$ for the full modular group
$ \mathrm{SL}_2(\mathbb Z)$. For each $f\in\mathcal{H}_k$, let $\lambda_f(n)$ 
denote the corresponding eigenvalue of the normalized Hecke operator
$T_n$. We prove nontrivial upper bounds for  \[\sum_{f \in \mathcal{H}_k}\omega_f \left|\sum_{n \leq x} \lambda_f(n)\right|^{2q},\] 
where $k$ is a  positive even integer, $1\leqslant x\leqslant k$, $0\leqslant q \leqslant 1$ and $
        \omega_f
        =
        \frac{\Gamma(k-1)}
        {(4\pi)^{k-1}\langle f,f\rangle}
        =
        \frac{2\pi^2}
        {(k-1)L(1,\mathrm{Sym}^2 f)}.
$ Our estimates match the conjecturally sharp upper bound for $1\leqslant x\leqslant k^{0.99}$. In particular, whenever both $x$ and $k/x$ tend to infinity with $k$, we obtain  \[\sum_{f \in \mathcal{H}_k}\omega_f \left|\sum_{n \leq x} \lambda_f(n)\right|=o(\sqrt{x}).\]   We  also prove upper bounds for low moments
of sums of
Hecke eigenvalues  associated with Hecke--Maass
cusp forms for $\mathrm{SL}_2(\mathbb Z)$.

We study these sums through the probabilistic model proposed by Cogdell--Michel. We also determine the order of magnitude of low moments of this probabilistic model, which is equivalent to determining the order of magnitude  of low moments of Hecke eigenvalue sums in the limit.
\end{abstract}

\maketitle

\begin{flushright}
\dedicatory{{\it \small }}
\end{flushright}
\setcounter{tocdepth}{1}

%\vglue -12mm
%\tableofcontents

%\vglue -15mm
\section{Introduction}

Understanding sums of arithmetic functions is a classical and important topic in number theory. Among these, character sums and zeta sums have been extensively studied.
Under GRH, for any non-trivial Dirichlet character \(\chi\bmod r\) one has 
\[
\sum_{n \leq x} \chi(n)\ll_\varepsilon x^{1/2+\varepsilon}r^\varepsilon .
\]
However, Harper’s breakthrough work \cite{H23} showed that character sums and zeta sums typically exhibit ``better than squareroot cancellation''. More precisely, uniformly for any $1 \leq x \leq r$ and any $0 \leq q \leq 1$, he proved the  upper bound
\begin{equation}\label{eq:upperbd-lowmoment-character-sums}
    \frac{1}{r-1} \sum_{\chi \bmod r} \left|\sum_{n \leq x} \chi(n)\right|^{2q} \ll \Biggl(\frac{x}{1 + (1-q)\sqrt{\log\log(10L_r)}} \Biggr)^q , \end{equation} 
where $r$ is a prime and $ L_r := \min\{x,r/x\}$. Taking $q=1/2$, we have, provided that \(L_r\to\infty\), 
\begin{equation}
    \frac{1}{r-1} \sum_{\chi \bmod{r} } \left|\sum_{n \leq x} \chi(n)\right|=o(\sqrt{x})  , \end{equation} which implies that the typical size of character sums is $o(\sqrt{x}).$  He also proved an analogous estimate for zeta sums: uniformly for any $1 \leq x \leq T$ and any $0 \leq q \leq 1$,
\begin{equation}\label{eq:upperbd-lowmoment-zeta-sums}
 \frac{1}{T} \int_{0}^{T} \left|\sum_{n \leq x} n^{it}\right|^{2q} \,dt \ll \Biggl(\frac{x}{1 + (1-q)\sqrt{\log\log(10L_T)}} \Biggr)^q ,
\end{equation}
where $L_T:=\min\{x,T/x\}$. In a recent paper \cite{H26} and in forthcoming work \cite{H}, Harper established
matching lower bounds for low moments of character sums, showing
that the above estimates \eqref{eq:upperbd-lowmoment-character-sums}  are sharp. 
We refer readers to \cite{GW25, X24, WX25, HX26} for recent developments related to Harper's work.

From the viewpoint of symmetry types of families of $L$-functions, the family $L(s,\chi)$ with $\chi$ rangeing over primitive Dirichlet characters of a fixed modulus and  the family  \(\zeta(s+it)\) with \(0 \leq t \leq T\) are two examples of unitary families of $L$-functions. Thus, Harper's work may be viewed as establishing the ``better than squareroot cancellation'' phenomenon for coefficient sums of $L$-functions arising from unitary families. It is thus a natural and important question to understand if this striking phenomenon is universal or not.  The main purpose of the present paper is to show that the same phenomenon persists for the orthogonal families of $L$-functions, specifically the orthogonal
families of holomorphic cusp forms of weight $k$ for \(\mathrm{SL}_2(\mathbb Z)\) and Hecke--Maass cusp forms for \(\mathrm{SL}_2(\mathbb Z)\). We will also briefly discuss the symplectic family of quadratic
characters later in the paper.
\bigskip
 
\noindent\emph{Orthogonal families.}
Let \(k\) be a positive even integer, and let \(\mathcal{H}_k\) denote the set
of normalized holomorphic Hecke cusp forms  of weight \(k\) for
\(\mathrm{SL}_2(\mathbb Z)\). For each \(f\in\mathcal{H}_k\), let \(\lambda_f(n)\) 
denote the corresponding eigenvalue of the normalized Hecke operator
\(T_n\) $(n\geq 1)$. Each \(f\in \mathcal{H}_k\) has the  Fourier expansion
\[
        f(z)=\sum_{n=1}^{\infty}
        \lambda_f(n)n^{(k-1)/2}e(nz),
\]
with \(\lambda_f(1)=1\). The classical dimension formula gives
\[
    |\mathcal{H}_k|
    = \frac{k}{12}+O(1).
\]   The  Hecke eigenvalues \(\lambda_f(n)\) are real and
satisfy the Hecke relation
\[
        \lambda_f(m)\lambda_f(n)
        =
        \sum_{d\mid (m,n)}
        \lambda_f\left(\frac{mn}{d^2}\right).
\]
We are interested in the sums of Hecke eigenvalues
\[
        S_f(x):=\sum_{n\leq x}\lambda_f(n).
\]

For a fixed form $f$, estimates for $S_f(x)$ have been obtained by various authors (see \cite{H27,HI89,W09}). The best known bound is due to Rankin \cite{R89}. He showed that
\[
    \sum_{n\leqslant x}\lambda_f(n)
    \ll_f x^{1/3}(\log x)^{-(1-8/(3\pi))}
\]
assuming the Sato--Tate conjecture, which is now a theorem of
Barnet-Lamb, Geraghty, Harris and Taylor \cite{BLGHT11}. 
Note that the implicit constant here depends on $f$ and thus depends on $k$. Assuming GRH, one can show that uniformly for all $f\in \mathcal{H}_k$,
\[
S_f(x)\ll_\varepsilon x^{1/2+\varepsilon}k^\varepsilon .
\]

Relatively little is known about the average behavior of \(S_f(x)\) as
\(f\) ranges over \(\mathcal{H}_k\). To the best of our knowledge, the work most closely related to ours is that of Carmichael~\cite{C26a,C26b}, who investigated the first moment, without absolute values, and the second
moment  of the dyadic sums
\(\sum_{x\leq n\leq 2x}\lambda_f(n)\)
over \(f\in\mathcal{H}_k\) of large weight. By applying the  Vorono\"i summation formula, Carmichael also proved that for $x\geq k^2/(8\pi^2)$ and  for any $\varepsilon>0$,  one has, uniformly for $f\in \mathcal{H}_k$,
\[              \sum_{x\leq n\leq 2x}\lambda_f(n)\ll x^{1/3+\varepsilon}.
\] 

In this paper we shall be concerned with the typical size
of \(S_f(x)\) as \(f\) varies in \(\mathcal{H}_k\).  We shall use
the harmonic weights that arise naturally in the Petersson trace formula.  For
\(f\in\mathcal{H}_k\), define the harmonic weight
\[
        \omega_f
        :=
        \frac{\Gamma(k-1)}
        {(4\pi)^{k-1}\langle f,f\rangle}
        =
        \frac{2\pi^2}
        {(k-1)L(1,\mathrm{Sym}^2 f)}
\]
where $\langle f,f\rangle$ denotes the Petersson inner product.
For any function \(W\) on \(\mathcal H_k\), set
\[
        \sumh_{f\in\mathcal{H}_k}W(f)
        :=
        \sum_{f\in\mathcal{H}_k}\omega_fW(f) .
\]
For a subset \(S\subseteq\mathcal{H}_k\), we define its harmonic measure by
\[
|S|_h:=\sumh_{f\in S}1.
\]
By the Petersson trace formula (see Lemma \ref{lemma:orthogonality-holomorphic}), we have
\[
        |\mathcal{H}_k|_h=1+O(e^{-k}).
\]
This shows that the harmonic measure is asymptotically a probability measure on
\(\mathcal{H}_k\).  Moreover, we have 
\[
        \frac{1}{k\log k}\ll \omega_f\ll \frac{\log k}{k}
\]
(see Goldfeld, Hoffstein and Lieman \cite[Appendix]{HL94}), from which we see that the harmonic weights are comparable to the uniform weight $1/|\mathcal{H}_k|\asymp 1/k$ up to logarithmic factors.

The Petersson trace formula implies that for \(1\leq x\leq k/100\),
\[
        \sumh_{f\in\mathcal{H}_k}
        \left|
        \sum_{n\leq x}\lambda_f(n)
        \right|^2
        =
        \lfloor x\rfloor+O(e^{-k}x^2)\ll x.
\]
It follows from H\"older's inequality that
\[
        \sideset{}{^h}\sum_{f\in\mathcal{H}_k}
        \left|
        \sum_{n\leq x}\lambda_f(n)
        \right|^{2q}
        \ll x^q
        \qquad (0\leq q\leq 1,~1\leq x\leq k/100).
\] 

Our first set of theorems improves this trivial bound.

 \begin{corollary}Let $k$ be a large positive even integer.
When both $x$ and $k/x$ tend to infinity with $k$, 
\begin{equation}\label{hecke-betterthan-square-root}\sumh_{f \in \mathcal{H}_k} \left|\sum_{n \leq x} \lambda_f(n)\right|=o(\sqrt{x}).\end{equation} 
\end{corollary} 

We in fact prove the following quantitative results.
\begin{theorem}\label{mainthm:low-moment-modular-sharp}
Let $k$ be a large positive even integer.
Uniformly for $1 \leq x \leq k
\exp\!\left(-(\log\log k)^2\right)
$ and  $0 \leq q \leq 1$, we have
\begin{equation}\label{eq:mainthm-low-moment-modular-sharp}
\sumh_{f \in \mathcal{H}_k} \left|\sum_{n \leq x} \lambda_f(n)\right|^{2q} \ll \Biggl(\frac{x}{1 + (1-q)\sqrt{\log\log(10L_k)}} \Biggr)^q , \end{equation} 
where $L_k :=\min\left\{x,~ k/x\right\}.
$
\end{theorem}

\begin{theorem}\label{mainthm:low-moment-modular}
Let $k$ be a large positive even integer.
Uniformly for any $ k
\exp\!\left(-(\log\log k)^2\right) \leq x \leq k$ and any $0 \leq q \leq 1$, we have
\begin{equation}\label{eq:mainthm:low-moment-modular}
\sumh_{f \in \mathcal{H}_k} \left|\sum_{n \leq x} \lambda_f(n)\right|^{2q} \ll \Biggl(\frac{x}{1 + (1-q)\sqrt{\log\log(100 \log(10k/x))}} \Biggr)^q .\end{equation} 
\end{theorem}

\begin{remark}
According to the Vorono\"i summation formula (see \cite[Lemma 2.3]{C26a}), there is a duality between the Hecke eigenvalue sums of length $x$ and $k^2/x$.  We expect the following to hold:
uniformly for $1 \leq x \leq k^2
$ and  $0 \leq q \leq 1$, we have
\begin{equation}\label{eq:mainthm-low-moment-modular-conjecture}
\sumh_{f \in \mathcal{H}_k} \left|\sum_{n \leq x} \lambda_f(n)\right|^{2q} \ll \Biggl(\frac{x}{1 + (1-q)\sqrt{\log\log(100\mathcal{L}_k)}} \Biggr)^q , \end{equation} 
where $\mathcal{L}_k:=\min\left\{x,~ k^2/x\right\}
$ (see \cite[the discussion following Theorems~1 and~2]{H23}). Thus,
\eqref{hecke-betterthan-square-root} should remain valid when both $x$ and $k^2/x$ tend to infinity with $k$.
\end{remark}
The following corollary of Theorem \ref{mainthm:low-moment-modular-sharp} shows quantitatively that, for most \(f\in \mathcal{H}_k\) (with respect to harmonic measure), the sum of Hecke eigenvalues exhibits  ``better than square-root cancellation''.

\begin{corollary}\label{cor:quantitative-corollary}
Let $k$ be a large positive even integer. Let
\[\mathcal{E}=\left\{
    f\in \mathcal{H}_k:
    \left|\sum_{n\leq x}\lambda_f(n)\right|
    \geq \lambda\frac{\sqrt{x}}{(\log\log(10L_k))^{1/4}}
\right\}\]Then uniformly for  $1\leq x\leq k
\exp\!\left(-(\log\log k)^2\right)$
and any $\lambda\geq 2$, we have
\[
|\mathcal{E}|_h\ll
\frac{\min\left\{\log\lambda,\sqrt{\log\log(10L_k)}\right\}}
     {\lambda^2},
\]
where $L_k=\min\{x,k/x\}$.
\end{corollary}
We remark that in an independent work, Hamdan, Leung  and Wong \cite{HLW26} proved the “better than squareroot cancellation” phenomenon for sums of Hecke eigenvalues of holomorphic cusp forms, but in a different family, where they fix the weight $k$ and let the level $N$ grow.
\medskip

We next consider the  orthogonal family of Hecke--Maass
cusp forms for $\mathrm{SL}_2(\mathbb{Z})$. 
Let $\{u_j\}_{j\geq1}$ be an orthonormal basis of Hecke--Maass cusp forms
for $\mathrm{SL}_2(\mathbb{Z})$. Let $
\frac{1}{4}+t_j^2
$ with $t_j>0$ be the Laplace eigenvalue of $u_j$. For each \(u_j\), let \(\lambda_j(n)\)  denote the corresponding eigenvalue of the normalized Hecke operator \(T_n\) $(n\geq 1)$. Each $u_j$  has the following Fourier expansion:
\[
u_j(x+iy)
=
\sum_{n\neq 0}
\rho_j(n)\sqrt{y}\,K_{it_j}(2\pi |n|y)e(nx),
\]
where $K_{\nu}(x)$ is the $K$-Bessel function. We have
\(
\rho_j(\pm n)=\rho_j(\pm1)\lambda_j(n)\)
$(n\geq1)$
and
\(
|\rho_j(\pm 1)|^2
=
\frac{2\cosh(\pi t_j)}
{L(1,\operatorname{sym}^2u_j)}.
\)
\begin{theorem}\label{mainthm:low-moment-maass}
Let \(T\) be a large real number.  Uniformly for \(1\leq x\leq T^{0.999}\) and
\(0\leq q\leq 1\), we have
\begin{equation}\label{eq: low-moment-maass}
        \frac{12}{T^2}\sum_{0\leq t_j\leq T}
        \frac{\zeta(2)}{L(1,\mathrm{sym}^2 u_j)}\left|
        \sum_{n\leq x}\lambda_j(n)
        \right|^{2q}
        \ll
        \left(
        \frac{x}
        {1+(1-q)\sqrt{\log\log(100x)}}
        \right)^q.
\end{equation}
\end{theorem}

\begin{remark}
Let \(N(T):=\sum_{t_j\leq T}1\) denote the number of Hecke--Maass eigenforms with spectral
parameter \(t_j \leq T\). By Weyl's law,
\[
    N(T) \sim \frac{T^2}{12}.
\]
%In other words, there are about
% \frac{1}{12}T^2 
%
%linearly independent $eigenforms with $eigenvalue $t_j\leq T$.
  In \cite{L11}, 
  Xiaoqing Li proved the following weighted Weyl law: $$\sum_{t_j\leq T}
        \frac{\zeta(2)}{L(1,\mathrm{sym}^2 u_j)}=\frac{T^2}{12}+O(T).$$
Therefore, \eqref{eq: low-moment-maass} is equivalent to \begin{equation}\label{eq: low-moment-maass-equivalent}
        \frac{1}{\sum_{t_j\leq T}
        \frac{\zeta(2)}{L(1,\mathrm{sym}^2 u_j)}}\sum_{t_j\leq T}
        \frac{\zeta(2)}{L(1,\mathrm{sym}^2 u_j)}\left|
        \sum_{n\leq x}\lambda_j(n)
        \right|^{2q}
        \ll
        \left(
        \frac{x}
        {1+(1-q)\sqrt{\log\log(100x)}}
        \right)^q .
\end{equation}
\end{remark}
\medskip

\noindent\emph{Symplectic families}
 Now we briefly discuss moments of quadratic character sums.
One might expect that a similar ``better than squareroot cancellation" should happen in quadratic character families. Namely 
\begin{equation}\label{mainthm-quadratic}
        \frac{1}{D}
        \sideset{}{^{\flat}}\sum_{|d|\leq D}
        \left|
    \sum_{\substack{n\leq x\\ n\ \mathrm{squarefree}}}
        \chi_d(n)
        \right|^{2q}
        \ll
        \left(
        \frac{x}
        {1+(1-q)\sqrt{\log\log(100x)}}
        \right)^q,
\end{equation}
where $0 \leq q \leq 1$,  $\chi_d(\cdot):=(\frac{d}{\cdot})$ denotes the Kronecker symbol and $\sum^{\flat}$ denotes summation over fundamental discriminants.
Similar to our previous results, the relative size between the conductor $D$ and summation length $x$ is important here. Roughly speaking, the smaller $x$ is relative to $D$,
the more the character sums behave like random sums
and the easier it would be to establish the above estimates.  We can show that \eqref{mainthm-quadratic} holds for $x\ll D^{4/11-\varepsilon}$ unconditionally and $x\ll D^{1/2-\varepsilon}$ under GRH using the probabilistic model introduced by Granville-Soundararajan \cite{GS03}. We note that in a forthcoming work, Nosal is able to achieve such a result unconditionally for the range $x\ll D^{1/2-\epsilon}$ in the prime conductor setting. We believe his proof should be able to be extended to the fundamental discriminant case as well. We skip our proof of \eqref{mainthm-quadratic} in view of his stronger results.        

\subsection{Probabilistic models and remarks on the proof}
The proofs of Theorems \ref{mainthm:low-moment-modular-sharp}, \ref{mainthm:low-moment-modular},
\ref{mainthm:low-moment-maass} adapt the method 
developed by Harper in \cite{H23}. Harper's treatment of the low moments of character sums in
\cite{H23} builds on his earlier work \cite{H20} on the low moments of
Steinhaus random multiplicative functions. One may view his proof as a derandomization strategy. Roughly speaking, Harper transfers the character sum averages,
taken over all characters \(\chi \bmod r\), to certain
averages of the Steinhaus random multiplicative function. This allows him to apply
 \cite[Multiplicative Chaos Result 1]{H23} arising in the proof of the upper bound for low moments of Steinhaus random multiplicative functions. More precisely, the
input needed is an upper bound for the low moments of integrals of random Euler
products over short intervals, which was obtained using the probabilistic theory of
critical multiplicative chaos. Thus, in order to prove our results, we
shall look for suitable probabilistic models, which can model the behavior of Hecke eigenvalues. 

For Theorems \ref{mainthm:low-moment-modular-sharp}, \ref{mainthm:low-moment-modular} and
\ref{mainthm:low-moment-maass}, we use the probabilistic model introduced by Cogdell
and Michel \cite{CM04} in their study of complex moments at $s=1$ of
symmetric power $L$-functions associated with holomorphic cusp forms.
Let \((\theta(p))_{p}\) be a sequence of
independent random variables,
each taking values in \([0,\pi]\) and distributed
with respect to the Sato--Tate measure 
\(
\mathrm{d}\mu_{\mathrm{s t}}(\theta)
        :=
        \frac{2}{\pi}\sin^2\theta\,\mathrm{d}\theta.
\)
For each prime \(p\) and integer \(a\geq 0\), set
\[
        \mathbb{X}(p):=2\cos\theta(p),
\]
and
\[
        \mathbb{X}(p^a)
        :=
        \sum_{j=0}^{a} e^{i(a-2j)\theta(p)}
        =
        \frac{\sin((a+1)\theta(p))}{\sin\theta(p)}.
\]
We define 
\[
        \mathbb{X}(n):=\prod_{p^a\Vert n}\mathbb{X}(p^a).
\]
Note that
 we have (see Lamzouri \cite[Lemma 3.2]{L19})
\begin{align*}      \mathbb{E}\bigl(\mathbb{X}(m)\mathbb{X}(n)\bigr)
        =
        \mathbf{1}_{m=n}.
\end{align*}
We determine the order of magnitude of the low moments of this
\(\GL_2\) random multiplicative function. 

\begin{theorem}\label{mainthm:low-moment-random-GL2}
Let \(\mathbb{X}(n)\) be the \(\GL_2\) random multiplicative function defined
above.  Uniformly for all large \(x\) and all \(0\leq q\leq 1\), we have
\[
        \mathbb{E}\left|
        \sum_{n\leq x}\mathbb{X}(n)
        \right|^{2q}
        \asymp
        \left(
        \frac{x}{1+(1-q)\sqrt{\log\log x}}
        \right)^q .
\]
\end{theorem}

Theorem \ref{mainthm:low-moment-random-GL2} is a \(\GL_2\) analogue of the remarkable results of Harper \cite{H20} concerning the low moments of  Steinhaus and
Rademacher  random multiplicative functions.\footnote{It is not surprising that in \cite{HLW26}, they also independently established the same result for the random model (at least the upper bound), as they also use Harper’s derandomization strategy and this is a necessary first step.}
Harper's result on Steinhaus   random multiplicative functions implies
\[
\lim_{T\to\infty}\frac{1}{T}\int_0^T
\left|\sum_{n\leq x} n^{-it}\right|^{2q}\,dt
=
\lim_{\substack{r\to\infty\\ r\ \mathrm{prime}}}
\frac{1}{r-1}
\sum_{\chi \bmod r}
\left|\sum_{n\leq x}\chi(n)\right|^{2q}\asymp
        \left(
        \frac{x}{1+(1-q)\sqrt{\log\log x}}
        \right)^q,
\]
since if $\X(n)$ is a Steinhaus   random multiplicative function, then
\[
\mathbb{E}\left|\sum_{n\leq x} \X(n)\right|^{2q}
=
\lim_{T\to\infty}\frac{1}{T}\int_0^T
\left|\sum_{n\leq x} n^{-it}\right|^{2q}\,dt
=
\lim_{\substack{r\to\infty\\ r\ \mathrm{prime}}}
\frac{1}{r-1}
\sum_{\chi \bmod r}
\left|\sum_{n\leq x}\chi(n)\right|^{2q}.
\]
Similarly, we have the following consequences for the limiting behaviour of sums of  Hecke eigenvalues.
\begin{corollary}\label{Cor:limiting-behavior-of-sums of-Hecke-eigenvalues}
Uniformly for all large fixed  \(x\) and all \(0\leq q\leq 1\), we have
\begin{equation}\label{eq:limiting-behavior-of-sums of-holomorphic-Hecke-eigenvalues}
\lim_{\substack{k\to\infty\\2\mid k}} 
\sumh_{f \in \mathcal{H}_k}
\left| \sum_{n \leq x} \lambda_f(n) \right|^{2q}
\asymp\left(\frac{x}{1+(1-q) \sqrt{\log \log x}}\right)^q 
.
\end{equation}
\begin{equation}\label{eq:limiting-behavior-of-sums of-Maass-Hecke-eigenvalues}
\lim_{T \to \infty } 
\frac{12}{T^2}\sum_{t_j\leq T}
        \frac{\zeta(2)}{L(1,\mathrm{sym}^2 u_j)}\left|
        \sum_{n\leq x}\lambda_j(n)
        \right|^{2q}
\asymp\left(\frac{x}{1+(1-q) \sqrt{\log \log x}}\right)^q 
.
\end{equation}
\end{corollary}

Inspired by Harper's works \cite{H20,H23}, one may formulate the
following strategy for studying the typical size of coefficient
sums in a family \(\mathcal{F}\) of \(L\)-functions. First, one identifies
a suitable random multiplicative function that models the behavior of the coefficients of the \(L\)-functions in
\(\mathcal{F}\), and
establishes for this model an analogue of
\cite[Multiplicative Chaos Results 1]{H23}. Next, by using Harper's derandomization strategy, one transfers averages of coefficient sums, taken over the family \(\mathcal{F}\), to certain averages involving the specific random multiplicative function. Finally, one applies the multiplicative chaos estimate
obtained in the first step to these averages involving the random multiplicative function.

The applicability of this strategy relies on two aspects. On
the probabilistic side, an examination of the proof of
\cite[Theorems~1 and~2]{H20} suggests that the relevant
multiplicative chaos estimates are not specific to the Steinhaus or Rademacher random multiplicative function,
but should extend to more general random multiplicative functions
satisfying orthogonality relations and some additional conditions.  On the arithmetic side, the
coefficients in many natural families of \(L\)-functions
satisfy orthogonality relations.

We shall discuss, in the context of the proofs of Theorems~\ref{mainthm:low-moment-modular-sharp} and \ref{mainthm:low-moment-modular}, some technical points concerning the second  step described above. The \(\GL_2\) random multiplicative model introduced above captures the
behavior of Hecke eigenvalues. The next task is
therefore to pass from averages
of Hecke eigenvalue
sums to
certain averages involving \(\GL_2\) random multiplicative functions. In Harper's
work, \cite[Proposition~1]{H23} is the key proposition: it shows
that certain character averages are well approximated by the
corresponding averages of Steinhaus random multiplicative functions. A
crucial tool in its proof is the twisted second moment estimate of character sums
given in \cite[Lemma~1]{H23}.
Therefore, we need to establish, for holomorphic Hecke eigenvalues, an analogue of \cite[Lemma 1]{H23}.  The proof of that lemma, as well as the
proof of the analogous estimate for \(n^{it}\), exploits the complete
multiplicativity of \(\chi(n)\) and \(n^{it}\).
 By contrast, Hecke eigenvalues are not completely multiplicative. Therefore,  the arguments in the proof of \cite[Lemma~1]{H23} cannot be
applied directly to Hecke eigenvalue sums.

In this paper, we use two different approaches to estimate twisted
second moments of holomorphic Hecke eigenvalues
$$ \sumh_{f \in \mathcal{H}_k}\left|\sum_{n \leq x} c(n) \lambda_f(n) \right|^{2} \left|\sum_{p\le P}\left(
\frac{a(p)\lambda_f(p)}{\sqrt p}
+
\frac{a(p^2)(\lambda_f(p^2)-1)}{p}    \right)\right|^{2l}$$
 where $(c(n))_{n \leq x}$, $(a(p))_{p\leq P}$ and  $(a(p^2))_{p\leq P}$ are any complex numbers with $|c(n)|,|a(p)|,|a(p^2)|\leqslant 1$. The first approach is
used to prove Lemma~\ref{lemma: evenmomentlem-modular-1}. This lemma
is then applied in establishing
\eqref{eq:holomorphic-Eigenvalues-behave like-random-model1} in
Proposition~\ref{prop:holomorphic-Eigenvalues-behave like-random-model},
which in turn is used in the proof of
Theorem~\ref{mainthm:low-moment-modular-sharp}. We begin by passing to
the $\GL_2$ random model, at the cost of an acceptably small
error term. We then apply H\"older's inequality to separate the
contribution of the sum over \(n\leq x\) from that of the sum over
primes \(p\leq P\). The former can be handled relatively directly,
whereas the latter is estimated using an argument inspired by the proof of Khintchine's inequality.
The second approach is used to prove
Lemma~\ref{lemma: evenmomentlem-modular-2}. This lemma is applied in
establishing
\eqref{eq:holomorphic-Eigenvalues-behave like-random-model2} in
Proposition~\ref{prop:holomorphic-Eigenvalues-behave like-random-model},
which is used in the proof of
Theorem~\ref{mainthm:low-moment-modular}. In this
approach, we first bound the sum over primes \(p\leq P\) trivially and
then take the expectation of the remaining sum over \(n\leq x\).

The bound in Lemma~\ref{lemma: evenmomentlem-modular-1}
is of a similar shape  to the bound
\(x(\log P)(k!)\left(3\log\log P\right)^k\)
obtained from \cite[Lemma~1]{H23} by taking
\(\mathcal{P}=\{p\leq P\}\). This also explains that, for a certain range of the parameters, we can obtain a result of the same strength as \cite[Theorem~1]{H23}, namely Theorem~\ref{mainthm:low-moment-modular-sharp}.

We expect an analogue of Theorem \ref{mainthm:low-moment-maass}
to hold for Hecke--Maass forms for $\mathrm{SL}_n(\mathbb{Z})$ (\(n\geq 3\)). To this end, one may consider the following probabilistic model $\X_n$, which is a generalization of the \(\GL_2\) random model of Cogdell and Michel \cite{CM04}. Let $G:=\mathrm{SU}(n)$, equipped
with normalized Haar measure $\mu_G$. Let $G^\natural$ be the set of conjugacy classes of $G$, endowed  with the
Sato--Tate measure $\mu_{\mathrm{st},n}:=\pi_*\mu_G$, where
$\pi:G\to G^\natural$ is the natural projection. Let
$(g_p^\natural)_p$ be a sequence of
independent random variables indexed by the primes,
each taking values in $G^\natural$ 
and distributed
with respect to the Sato--Tate measure $\mu_{\mathrm{st},n}$. For each integer \(a\geq 0\), let \(\operatorname{Sym}^a\) denote the
\(a\)-th symmetric power of the standard representation of \(SU(n)\).
For \(g^\natural\in G^\natural\), we write
\[
\operatorname{tr}\bigl(\operatorname{Sym}^a(g^\natural)\bigr)
:=
\operatorname{tr}\bigl(\operatorname{Sym}^a(g)\bigr),
\quad g\in g^\natural,
\]
which is independent of the choice of the representative \(g\).
For each prime $p$, set
\[
\mathbb{X}_n(p^a)
:=
\operatorname{tr}\bigl(\operatorname{Sym}^a(g_p^\natural)\bigr).
\] In particular, \(\mathbb X_n(1)=1\). For $m=p_1^{a_1}\cdots p_r^{a_r}$, set\footnote{For \(n=2\), the probabilistic model presented here is exactly
the \(\GL_2\) random model of Cogdell and Michel \cite{CM04}.
The definition here is equivalent to the one given earlier,
since the map
\begin{align*}
\theta \longmapsto g(\theta)^\natural =
\begin{pmatrix}
e^{i\theta} & 0 \\
0 & e^{-i\theta}
\end{pmatrix}^{\natural}
\end{align*}
allows us to identify \([0,\pi]\) with \(G^\natural\)  and \(\mathrm{d}\mu_{\mathrm{st}}(\theta)
= \frac{2}{\pi}\sin^2(\theta)\,\mathrm{d}\theta\)
with \(\mu_{\mathrm{st},2}\).}
\[
\mathbb{X}_n(m)
:=
\prod_{j=1}^{r}\mathbb{X}_n(p_j^{a_j})
=
\prod_{j=1}^{r}
\operatorname{tr}\bigl(\operatorname{Sym}^{a_j}(g_{p_j}^\natural)\bigr).
\]
The representations \(\operatorname{Sym}^a\) (\(a\geq 0\)) are
irreducible and pairwise non-isomorphic (see Sepanski \cite{S07}). By Schur orthogonality,
\[
\int_{SU(n)}
\operatorname{tr}\bigl(\operatorname{Sym}^{a}(g)\bigr)\overline{\operatorname{tr}\bigl(\operatorname{Sym}^{b}(g)\bigr)}\,d\mu_G(g)
=\delta_{a,b}.
\]
Since \(\mu_{\mathrm{st},n}=\pi_*\mu_G\), it follows that
\[
\mathbb E\!\left(
\mathbb X_n(p^a)\overline{\mathbb X_n(p^b)}
\right)
=\delta_{a,b},
\]
which, together with the independence of the random variables \(g_p^\natural\) for distinct primes, gives 
\[
\mathbb E\!\left(
\mathbb X_n(m)\overline{\mathbb X_n(\ell)}
\right)
=\delta_{m,\ell}.
\]
In view of this orthogonality relation, it is believable that one can establish an analogue of Theorem \ref{mainthm:low-moment-random-GL2} for this probabilistic model. Moreover, it should be possible to use the strategy described above to establish  an analogue of Theorem \ref{mainthm:low-moment-maass}
for Hecke--Maass forms for $\mathrm{SL}_n(\mathbb{Z})$  (\(n\geq 3\)). 
\subsection*{Notation and convention}
For any function $W$ on $\mathcal{H}_k$, we define \[\E^{\rm holo}W(f):=\sum_{f\in\mathcal{H}_k}\omega_fW(f).\] For any function $W$ on $\{u_j\}_{j\geq1}$, we define \[\E^{\rm Maass}W(u_j):=\frac{6}{T^2}\sum_{j}
        \frac{\zeta(2)}{L(1,\mathrm{sym}^2 u_j)}e^{-t_j/T}W(u_j).\] 

\subsection*{Acknowledgements}
We are very grateful to Adam Harper, Alexandre de Faveri, Andrew Granville, Didier Lesesvre, Junxian Li, Pawe{\l}~Nosal and  Ping Xi for helpful discussions. The second author would like to thank Didier Lesesvre for pointing out the reference \cite{L11}. Part of the work was done when both authors were participating in the 2026 thematic program \emph{Universal Statistics in Number Theory} at the Centre de recherches mathématiques
(CRM) in Montreal, and we thank the organizers for their hospitality.  The second author carried out part of this work while participating
in the Summer School \emph{``FOUVRY-73''}, and he  would like to thank the organizers for their hospitality. The second author thanks Chengjie Wang for  helpful discussions.

\section{preliminaries}

\subsection{Orthogonality relations}
We record the following consequence of the Petersson trace formula, which is stated as Lemma 2.1 in \cite{RS06}.
Let $\mathcal{H}_k$ be defined as before. 
\begin{lemma}\label{lemma:orthogonality-holomorphic}
Let $k$ be large, and let $m$ and $n$ be natural numbers with
$
mn \leq \frac{k^2}{10^4}
$. Then
\[\sumh_{f \in \mathcal{H}_k} \lambda_f(m) \lambda_f(n) = \delta_{m=n} + O(e^{-k}).\]
\end{lemma}

The following lemma follows from a standard application of the Kuznetsov formula, which is stated as Lemma 1 in \cite{BBR14}. 

\begin{lemma}\label{lemma:orthogonality-maass}
For \(m,n\in \mathbb{N}\), \(T\geq 1\) and \(\varepsilon>0\) we have
\[
\sum_j 
\frac{\zeta(2)}{L(1,\operatorname{sym}^2 u_j)}
e^{-t_j/T}\lambda_j(m)\lambda_j(n)
=
\delta_{m=n}\frac{T^2}{6}
+
O_\varepsilon\left(
T(mnT)^\varepsilon+(mn)^{1/4+\varepsilon}
\right).
\]
\end{lemma}
\subsection{A smooth partition of unity}\label{subsecpartition}
In the process of using Harper's derandomization strategy, we shall need the following result about a smooth partition of unity, which plays a role like the conditioning argument in the random case.

\begin{lemma}[Harper {\cite[Approximation Result~1]{H23}}]\label{lemma:partition-of-unity}
Let $N \in \N$ be large, and $\delta > 0$ be small. There exist functions $g : \R \rightarrow \R$ (depending on $\delta$) and $g_{N+1} : \R \rightarrow \R$ (depending on $\delta$ and $N$) such that, if we define $g_{j}(x) = g(x - j)$ for all integers $|j| \leq N$, we have the following properties:
\begin{enumerate}
\item $\sum_{|j| \leq N} g_{j}(x) + g_{N+1}(x) = 1$ for all $x \in \R$;

\item $g(x) \geq 0$ for all $x \in \R$, and $g(x) \leq \delta$ whenever $|x| > 1$;

\item $g_{N+1}(x) \geq 0$ for all $x \in \R$, and $g_{N+1}(x) \leq \delta$ whenever $|x| \leq N$;

\item for all $l \in \N$ and all $x \in \R$, we have the derivative estimate $|\frac{d^{l}}{dx^{l}} g(x)| \leq \frac{1}{\pi (l+1)} (\frac{2\pi}{\delta})^{l+1}$.

\end{enumerate}
\end{lemma}
\section{\texorpdfstring{\(\GL_2\)}{GL2} Probabilistic models}
\subsection{Probabilistic models}
Let \((\theta(p))_{p}\) be a sequence of
independent random variables with values in \([0,\pi]\), each distributed
with respect to the Sato--Tate measure
\(
        d\mu_{st}(\theta)=\frac{2}{\pi}\sin^2\theta\,d\theta .
\)
For a prime \(p\), set
\[
        \mathbb{X}(p)=2\cos\theta(p),
\]
and, more generally,
\[
        \mathbb{X}(p^a)
        =
        \sum_{j=0}^{a} e^{i(a-2j)\theta(p)}
        =
        \frac{\sin((a+1)\theta(p))}{\sin\theta(p)}.
\]
We define $\X(1)=1$ and extend \(\mathbb{X}\) multiplicatively by setting
\[
        \mathbb{X}(n):=\prod_{p^a\Vert n}\mathbb{X}(p^a).
\]
One can easily verify that
the \(\mathbb{X}(n)\) satisfy the
Hecke relations, namely that
\[
\mathbb{X}(m)\mathbb{X}(n)
=
\sum_{d\mid(m,n)}
\mathbb{X}\left(\frac{mn}{d^2}\right).
\]
Note that we have (see Lamzouri \cite[Lemma 3.2]{L19})
\begin{equation}\label{eq:orthoganol-relation-GL2-RMf}
        \mathbb{E}\bigl(\mathbb{X}(m)\mathbb{X}(n)\bigr)
        =
        \mathbf{1}_{m=n}.
\end{equation}
 The following Lemmas \ref{lemma: Eigenvalue like random model-holomorphic} and \ref{lemma: Eigenvalue like random model-maass} show that the $\GL_2$ Random multiplicative function captures the behavior of Hecke eigenvalues in a natural and direct sense.
\begin{lemma}Let $\X$ be a $\GL_2$ Random multiplicative function. For $n_1,\ldots,n_r\in \mathbb{N}$ with $n_1n_2\cdots n_r\leq k^2/10^4$, we have
\label{lemma: Eigenvalue like random model-holomorphic}
    \[
\E^{\rm holo}\bigl(\lambda_f(n_1)\lambda_f(n_2)\cdots \lambda_f(n_r)\bigr)
=\mathbb{E}\bigl(\X(n_1)\X(n_2)\cdots \X(n_r)\bigr)
+O((\tau(n_1)\cdots \tau(n_r))^{2}e^{-k}).
\]
\end{lemma}

\begin{proof}
The
Hecke relations  can be written as
\[
\lambda_f(n_1)\lambda_f(n_2)
=
\sum_{m\mid n_1n_2} b_m(n_1,n_2)\lambda_f(m),
\]
where \(b_m(n_1,n_2)=1\) if \(m=n_1n_2/d^2\) for some \(d\mid(n_1,n_2)\),
and equals \(0\) otherwise. More generally, one can write
\begin{equation}\label{eq:
multi-heckerelation1}
\lambda_f(n_1)\cdots \lambda_f(n_r)
=
\sum_{m\mid \prod_{j=1}^{r} n_j}
b_m(n_1,\ldots,n_r)\lambda_f(m),
\end{equation}
for some non-negative integers \(b_m(n_1,\ldots,n_r)\). Moreover, one can easily prove  that
\begin{equation}\label{eq:
bound for b_1}
b_m(n_1,\ldots,n_r) \leq \tau(n_1)\cdots \tau(n_r).
\end{equation}
Using \eqref{eq:
multi-heckerelation1}, \eqref{eq: bound for b_1} and Lemma \ref{lemma:orthogonality-holomorphic} ,  we have
\begin{equation}
\begin{aligned}
\E^{\rm holo}\bigl(\lambda_f(n_1)\lambda_f(n_2)\cdots \lambda_f(n_r)\bigr)
&=
\sum_{m\mid n_1n_2\cdots n_r}
b_m(n_1,n_2,\ldots,n_r)\E^{\rm Holo}\bigl(\lambda_f(m)\bigr)\\
&=
b_1(n_1,n_2,\ldots,n_r)+O(( \tau(n_1)\cdots \tau(n_r))^{2}e^{-k}).
\end{aligned}
\end{equation}
Similarly, we have
    \[
\mathbb{E}\bigl(\X(n_1)\X(n_2)\cdots \X(n_r)\bigr)
=
\sum_{m\mid n_1n_2\cdots n_r}
b_m(n_1,n_2,\ldots,n_r)\mathbb{E}\bigl(\X(m)\bigr)
=
b_1(n_1,n_2,\ldots,n_r).
\]
\end{proof}

\begin{lemma}Let $\X$ be a $\GL_2$ Random multiplicative function. For $n_1,\ldots,n_r\in \mathbb{N}$, we have 
\label{lemma: Eigenvalue like random model-maass}
\begin{align*}
&\mathbb{E}^{\rm Maass}\bigl(\lambda_j(n_1)\lambda_j(n_2)\cdots \lambda_j(n_r)\bigr)\\
=&\mathbb{E}\bigl(\X(n_1)\X(n_2)\cdots \X(n_r)\bigr)\\
&+O_{\varepsilon}((\tau(n_1)\cdots \tau(n_r))^{2}( n_1n_2\cdots n_r)^{\epsilon}T^{-(1-\varepsilon)}+(\tau(n_1)\cdots \tau(n_r))^{2}( n_1n_2\cdots n_r)^{1/4+\varepsilon}T^{-2}).
\end{align*}
\end{lemma}
\begin{proof}
    The proof is similar to the proof of Lemma \ref{lemma: Eigenvalue like random model-holomorphic} with  Lemma \ref{lemma:orthogonality-holomorphic} replaecd by  Lemma \ref{lemma:orthogonality-maass} 
\end{proof}

\subsection{\(\GL_2\) random Euler products.} 

The following lemmas concern mixed moments of certain Euler products
associated with the \(\GL_2\) random multiplicative function.  A basic
quantity in these calculations is
\[
R_p(t):=-\Re\log\left(
1-\frac{\X(p)}{p^{1/2+\sigma+it}}
+\frac{1}{p^{1+2\sigma+2it}}
\right),
\]
which is the logarithm of the modulus of the local factor at \(p\).
 It turns out that the estimates for
\[
\mathbb{E}R_p(t),\qquad
\mathbb{E}R_p(t)^2,\qquad
\mathbb{E}R_p(t_1)R_p(t_1+t_2)
\]
agree with their counterparts of the Rademacher random multiplicative
function. This agreement explains why these mixed moment calculations for the \(\GL_2\) random Euler products yield results of the
same form as those in the Rademacher case.
\begin{lemma}\label{lemma:mean-square-randomEulerproducts-upperbd}
If $\X$ is a $\GL_2$ random multiplicative function, then for any real $t_1$, $t_2$ and $u$, any real $40000\left(1+u^2\right) \leqslant x \leqslant y$ and any real $\sigma \geqslant-1 / \log y$, we have
$$
\begin{aligned}
 &\ \ \ \ \mathbb{E}  \prod_{x<p \leqslant y}\left|1-\frac{\X(p)}{p^{1 / 2+\sigma+it_1}}+\frac{1}{p^{1+2\sigma+2it_1}}\right|^{-2}\left|1-\frac{\X(p)}{p^{1 / 2+\sigma+i (t_1+t_2)}}+\frac{1}{p^{1+2\sigma+2i(t_1+t_2)}}\right|^{-i u} \\&
= \exp \left\{\sum_{x<p \leqslant y} \frac{1+i u c\left(t_1, t_2, p\right)-\left(u^2 / 4\right)\left(1+\cos \left(2\left(t_1+t_2\right) \log p\right)\right)}{p^{1+2 \sigma}}+T(u)\right\}
\end{aligned}
$$
where $c\left(t_1, t_2, p\right)=2 \cos \left(t_1 \log p\right) \cos \left(\left(t_1+t_2\right) \log p\right)-(1 / 2) \cos \left(2\left(t_1+t_2\right) \log p\right)$, and $T(u)=T_{x, y, \sigma, t_1, t_2}(u)$ satisfies $|T(u)| \ll \frac{1+|u|^3}{\sqrt{x} \log x}$, and $\left|T^{\prime}(u)\right| \ll \frac{1}{\sqrt{x} \log x}$ when $|u| \leqslant 1$.
\end{lemma}

\begin{proof} This lemma is a $\GL_2$ analogue of \cite[Lemmas 1,~2]{H20} and the proofs are similar.
Set $R_p(t):=-\Re \log \left(1-\frac{\X(p)}{p^{1 / 2+\sigma+it}}+\frac{1}{p^{1+2\sigma+2it}}\right).$
Thus, we may rewrite
\begin{align}\label{eq:rewrite-localfactor}
&\ \ \ \ \left|1-\frac{\X(p)}{p^{1 / 2+\sigma+it_1}}+\frac{1}{p^{1+2\sigma+2it_1}}\right|^{-2}\left|1-\frac{\X(p)}{p^{1 / 2+\sigma+i (t_1+t_2)}}+\frac{1}{p^{1+2\sigma+2i(t_1+t_2)}}\right|^{-i u}  \\  &=\exp \left\{2 R_p(t_1)+i u R_p(t_1+t_2)\right\} \nonumber =1+\sum_{j=1}^{\infty} \frac{\left(2 R_p(t_1)+i u R_p(t_1+t_2)\right)^j}{j!} .\nonumber
\end{align}
Note that for primes $y \geqslant p>x \geqslant 40000\left(1+u^2\right)$, and for $\sigma \geqslant-\frac{1}{\log y}$, we have \(p^{-\sigma}\leq e\). Thus, we have
\[
\left|\frac{\X(p)}{p^{1/2+\sigma+it}}- \frac{1}{p^{1+2\sigma+2it}}\right|< 1.
\]
Using the Taylor expansion of the logarithm, we have
$$
\begin{aligned}
R_p(t)&=\sum_{k=1}^{\infty} \frac{\Re\left(\frac{\X(p)}{p^{1/2+\sigma+it}}- \frac{1}{p^{1+2\sigma+2it}}\right)^k}{k }\\&=\Re\left(\frac{\X(p)}{p^{1/2+\sigma+it}}- \frac{1}{p^{1+2\sigma+2it}}\right)+\Re\left(\frac{\X(p)^2}{2p^{1+2\sigma+2it}}-\frac{\X(p)}{p^{3/2+3\sigma+3it}}+ \frac{1}{2p^{2+4\sigma+4it}}\right)\\ &\ \ \ +\sum_{k=3}^{\infty} \frac{\Re\left(\X(p)p^{-it}-p^{-(1/2+\sigma+2it)}\right)^k}{k p^{k(1/2+\sigma)}}\\
&=\frac{\X(p)\cos(t\log p)}{p^{1/2+\sigma}}-\frac{\cos(2t\log p)}{p^{1+2\sigma}}+\frac{\X(p)^2\cos(2t\log p)}{2p^{1+2\sigma}}\\&
\ \ \ -\frac{\X(p)\cos(3t\log p)}{p^{3/2+3\sigma}}+\frac{\cos(4t\log p)}{2p^{2+4\sigma}}+\sum_{k=3}^{\infty} \frac{\Re\left(\X(p)p^{-it}-p^{-(1/2+\sigma+2it)}\right)^k}{k p^{k(1/2+\sigma)}}
\end{aligned}
$$
Using $p^{1/2+\sigma}\geqslant\tfrac{p^{1/2}}{e}\geqslant 7$, we have $$
\begin{aligned}
\sum_{k=3}^{\infty} \frac{\Re\left(\frac{\X(p)}{p^{1/2+\sigma+it}}- \frac{1}{p^{1+2\sigma+2it}}\right)^k}{k }&=\sum_{k=3}^{\infty} \frac{\Re\left(\X(p)p^{-it}-p^{-(1/2+\sigma+2it)}\right)^k}{k p^{k(1/2+\sigma)}}\\&\ll \frac{1}{p^{3/2+3 \sigma}} \sum_{k=3}^{\infty} \frac{3^k}{k7^{k-3}}\ll \frac{1}{p^{3/2+3 \sigma}},
\end{aligned} 
$$
It follows that 
\begin{align*}
R_p(t)=\frac{\X(p)\cos(t\log p)}{p^{1/2+\sigma}}-\frac{\cos(2t\log p)}{p^{1+2\sigma}}+\frac{\X(p)^2\cos(2t\log p)}{2p^{1+2\sigma}}+O\left(\frac{1}{p^{3/2+3 \sigma}}\right).
\end{align*}
By \eqref{eq:orthoganol-relation-GL2-RMf},  we see that
\begin{align}\label{eq:expectation-Rp(t)}
\E 
R_p(t)=-\frac{\cos(2t\log p)}{2p^{1+2\sigma}}+O\left(\frac{1}{p^{3/2+3 \sigma}}\right)
\end{align} 

\begin{align}\label{eq:expectation-Rp(t)2}
\mathbb{E} R_p(t)^2=\frac{\cos ^2(t \log p)}{p^{1+2 \sigma}}+O\left(\frac{1}{p^{3 / 2+3 \sigma}}\right)=\frac{1+\cos (2 t \log p)}{2 p^{1+2 \sigma}}+O\left(\frac{1}{p^{3 / 2+3 \sigma}}\right)
\end{align}

\begin{align}\label{eq:expectation-Rp(t1)Rp(t1+t2)}
\mathbb{E} R_p\left(t_1\right) R_p\left(t_1+t_2\right)=\frac{\cos \left(t_1 \log p\right) \cos \left(\left(t_1+t_2\right) \log p\right)}{p^{1+2 \sigma}}+O\left(\frac{1}{p^{3 / 2+3 \sigma}}\right).
\end{align}
For $j \geqslant 3$, we  have
\begin{align}\label{eq:Rp(t)-trivialbound}
\left|R_p(t)^j\right| \leqslant\left(\sum_{k=1}^{\infty} \frac{3^k}{p^{k(1 / 2+\sigma)}}\right)^j=\frac{3^j}{\left(p^{1 / 2+\sigma}-3\right)^j}.
\end{align}

Next, for primes $y \geqslant p>x \geqslant 40000\left(1+u^2\right)$ and $\sigma \geqslant-\frac{1}{\log y}$, we have $\frac{1}{p^{1 / 2+\sigma}} \leqslant \frac{e}{p^{1 / 2}}$, and hence $(2+|u|) / p^{1 / 2+\sigma} \leqslant 3e / 100$. From \eqref{eq:rewrite-localfactor},  \eqref{eq:expectation-Rp(t)}, \eqref{eq:expectation-Rp(t)2}, \eqref{eq:expectation-Rp(t1)Rp(t1+t2)} and \eqref{eq:Rp(t)-trivialbound},  it follows that
$$
\begin{aligned}
&\mathbb{E}\left|1-\frac{\X(p)}{p^{1 / 2+\sigma+it_1}}+\frac{1}{p^{1+2\sigma+2it_1}}\right|^{-2}\left|1-\frac{\X(p)}{p^{1 / 2+\sigma+i (t_1+t_2)}}+\frac{1}{p^{1+2\sigma+2i(t_1+t_2)}}\right|^{-i u}\\
= & 1-\frac{\cos \left(2 t_1 \log p\right)}{p^{1+2 \sigma}}-\frac{(i u / 2) \cos \left(2\left(t_1+t_2\right) \log p\right)}{p^{1+2 \sigma}}+O\left(\frac{1+|u|}{p^{3 / 2+3 \sigma}}\right) \\
& +\frac{\left(1+\cos \left(2 t_1 \log p\right)\right)+2 i u \cos \left(t_1 \log p\right) \cos \left(\left(t_1+t_2\right) \log p\right)}{p^{1+2 \sigma}} \\
& -\frac{\frac{u^2}{4}\left(1+\cos \left(2\left(t_1+t_2\right) \log p\right)\right)}{p^{1+2 \sigma}}+O\left(\frac{1+u^2}{p^{3 / 2+3 \sigma}}\right) \\
& +\mathbb{E} \sum_{j=3}^{\infty} \frac{\left(2 R_p\left(t_1\right)+i u R_p\left(t_1+t_2\right)\right)^j}{j!} \\
= & 1+\frac{1+i u c\left(t_1, t_2, p\right)-\left(u^2 / 4\right)\left(1+\cos \left(2\left(t_1+t_2\right) \log p\right)\right)}{p^{1+2 \sigma}}+D_p(u),
\end{aligned}
$$
where $D_p(u)$ satisfies $\left|D_p(u)\right| \ll \frac{1+|u|^3}{p^{3 / 2+3 \sigma}} \ll \frac{1+|u|^3}{p^{3 / 2}}$ and its derivative satisfies $\left|D_p^{\prime}(u)\right| \ll \frac{1}{p^{3 / 2+3 \sigma}} \ll \frac{1}{p^{3 / 2}}$ for $|u| \leqslant 1$.  In the last step, we have used that $\mathbb{E} \sum_{j=3}^{\infty} \frac{\left(2 R_p\left(t_1\right)+i u R_p\left(t_1+t_2\right)\right)^j}{j!} = O(\sum_{j=3}^{\infty} \frac{\left(2+|u|\right)^j3^j}{j!(p^{1/2+\sigma}-3)^j} )$.

Finally, note that we may rewrite our conclusions as
$$
\begin{aligned}
&\mathbb{E} \left|1-\frac{\X(p)}{p^{1 / 2+\sigma+it_1}}+\frac{1}{p^{1+2\sigma+2it_1}}\right|^{-2}\left|1-\frac{\X(p)}{p^{1 / 2+\sigma+i (t_1+t_2)}}+\frac{1}{p^{1+2\sigma+2i(t_1+t_2)}}\right|^{-iu} \\
=& 1+\frac{1+i u c\left(t_1, t_2, p\right)-\left(u^2 / 4\right)\left(1+\cos \left(2\left(t_1+t_2\right) \log p\right)\right)}{p^{1+2 \sigma}}+D_p(u) \\
= & \exp \left\{\frac{1+i u c\left(t_1, t_2, p\right)-\left(u^2 / 4\right)\left(1+\cos \left(2\left(t_1+t_2\right) \log p\right)\right)}{p^{1+2 \sigma}}
+T_p(u)\right\}
\end{aligned}
$$
where $T_p(u)$ again satisfies $\left|T_p(u)\right| \ll \frac{1+|u|^3}{p^{3 / 2}}$ and also $\left|T_p^{\prime}(u)\right| \ll \frac{1}{p^{3 / 2}}$ for $|u| \leqslant 1$. Since $\X$ is independent on distinct primes, we then deduce

\begin{align*}
&\ \ \ \ \mathbb{E}  \prod_{x<p \leqslant y}\left|1-\frac{\X(p)}{p^{1 / 2+\sigma+it_1}}+\frac{1}{p^{1+2\sigma+2it_1}}\right|^{-2}\left|1-\frac{\X(p)}{p^{1 / 2+\sigma+i (t_1+t_2)}}+\frac{1}{p^{1+2\sigma+2i(t_1+t_2)}}\right|^{-iu}\\
& = \exp \left\{\sum_{x<p \leqslant y} \frac{1+i u c\left(t_1, t_2, p\right)-\left(u^2 / 4\right)\left(1+\cos \left(2\left(t_1+t_2\right) \log p\right)\right)}{p^{1+2 \sigma}}+\sum_{x<p \leqslant y} T_p(u)\right\}
\end{align*}
which implies Lemma \ref{lemma:mean-square-randomEulerproducts-upperbd} in view of the standard Chebychev-type estimate $\sum_{p>x} 1 / p^{3 / 2} \ll 1 /(\sqrt{x} \log x)$.
\end{proof}

\begin{lemma}\label{lemma:mean-square-randomEulerproducts-lowerbd}
 If $\X$ is a $\GL_2$ random multiplicative function, then for any real real $t_1, t_2, u, v$, any real $40000\left(1+u^2+v^2\right) \leqslant x \leqslant y$ and any real $\sigma \geqslant-1 / \log y$, we have
$$
\begin{aligned}
&\mathbb{E} \prod_{x<p \leqslant y}\left|1-\frac{\X(p)}{p^{1 / 2+\sigma+it_1}}+\frac{1}{p^{1+2\sigma+2it_1}}\right|^{-(2+iu)}\left|1-\frac{\X(p)}{p^{1 / 2+\sigma+i (t_1+t_2)}}+\frac{1}{p^{1+2\sigma+2i(t_1+t_2)}}\right|^{-(2+iv)} \\
= &\exp \left\{\sum_{x<p \leqslant y} \left(\frac{(1+i u / 2)^2+(1+i v / 2)^2}{p^{1+2 \sigma}}+\frac{(2+i u)(2+i v) \cos (t_2 \log p)}{2 p^{1+2 \sigma}}\right.\right.\\
&\qquad \qquad \qquad  +\frac{(2+i u)(i u / 2) \cos \left(2 t_1 \log p\right)+(2+i v)(i v / 2) \cos \left(2\left(t_1+t_2\right) \log p\right)}{2 p^{1+2 \sigma}} \\
&\qquad \qquad \qquad\left.\left.+\frac{(2+i u)(2+i v) \cos \left(\left(2 t_1+t_2\right) \log p\right)}{2 p^{1+2 \sigma}}\right)+T(u, v)\right\}
\end{aligned}
$$
where $T(u, v)=T_{x, y, \sigma, t}(u, v)$ satisfies $|T(u, v)| \ll \frac{1+|u|^3+|v|^3}{\sqrt{x} \log x}$ and its partial derivatives satisfy $\left|\frac{\partial T(u, v)}{\partial u}\right| \ll \frac{1+u^2+v^2}{\sqrt{x} \log x},\left|\frac{\partial T(u, v)}{\partial v}\right| \ll \frac{1+u^2+v^2}{\sqrt{x} \log x}$ and $\left|\frac{\partial T(u, v)}{\partial u \partial v}\right| \ll \frac{1+|u|+|v|}{\sqrt{x} \log x}$.
\end{lemma}

\begin{proof}
This lemma is a $\GL_2$ analogue of \cite[Lemma 6]{H20} (see also the computation in \cite[Section 5.3]{H20}) and can be proved in a similar way.

 As in lemma \ref{lemma:mean-square-randomEulerproducts-upperbd}, we  set $R_p(t):=-\Re \log \left(1-\frac{\X(p)}{p^{1 / 2+\sigma+it}}+\frac{1}{p^{1+2\sigma+2it}}\right).$
 Thus, we may rewrite
\begin{align*}
&\left|1-\frac{\X(p)}{p^{1 / 2+\sigma+it_1}}+\frac{1}{p^{1+2\sigma+2it_1}}\right|^{-(2+iu)}\left|1-\frac{\X(p)}{p^{1 / 2+\sigma+i (t_1+t_2)}}+\frac{1}{p^{1+2\sigma+2i(t_1+t_2)}}\right|^{-(2+iv)}  \\  =&\exp \left\{(2+iu) R_p(t_1)+(2+i v)R_p(t_1+t_2)\right\} \\
 =&1+\sum_{j=1}^{\infty} \frac{\left((2+iu) R_p(t_1)+(2+i v) R_p(t_1+t_2)\right)^j}{j!} .
\end{align*}
Next, note that for primes $y \geqslant p>x \geqslant 40000\left(1+u^2+v^2\right)$  and $\sigma \geqslant-\frac{1}{\log y}$, we have  $(4+|u|+|v|) / p^{1 / 2+\sigma} \leqslant 3e / 100$. From \eqref{eq:rewrite-localfactor},  \eqref{eq:expectation-Rp(t)}, \eqref{eq:expectation-Rp(t)2}, \eqref{eq:expectation-Rp(t1)Rp(t1+t2)} and \eqref{eq:Rp(t)-trivialbound},  it follows that
\begin{align*}&\mathbb{E}\left|1-\frac{\X(p)}{p^{1 / 2+\sigma+it_1}}+\frac{1}{p^{1+2\sigma+2it_1}}\right|^{-(2+iu)}\left|1-\frac{\X(p)}{p^{1 / 2+\sigma+i (t_1+t_2)}}+\frac{1}{p^{1+2\sigma+2i(t_1+t_2)}}\right|^{-(2+iv)}\\
= & 1-\frac{(2+iu)\cos \left(2 t_1 \log p\right)}{2p^{1+2 \sigma}}-\frac{(2+iv) \cos \left(2\left(t_1+t_2\right) \log p\right)}{2p^{1+2 \sigma}}+O\left(\frac{1+|u|+|v|}{p^{3 / 2+3 \sigma}}\right) \\
& +\frac{\left((1+iu/2)^2(1+\cos \left(2 t_1 \log p\right)\right)+(2 +i u )(2 +i v)\cos \left(t_1 \log p\right) \cos \left(\left(t_1+t_2\right) \log p\right)}{p^{1+2 \sigma}} \\
& +\frac{(1+iv/2)^2\left(1+\cos \left(2\left(t_1+t_2\right) \log p\right)\right)}{p^{1+2 \sigma}}+O\left(\frac{1+u^2+v^2}{p^{3 / 2+3 \sigma}}\right) \\
& +\mathbb{E} \sum_{j=3}^{\infty} \frac{\left((2+i u) R_p\left(t_1\right)+(2+i v) R_p\left(t_1+t_2\right)\right)^j}{j!} \\
= & 1+\frac{(1+i u / 2)^2+(1+i v / 2)^2}{p^{1+2 \sigma}}+\frac{(2+i u)(2+i v) \cos \left(t_2 \log p\right)}{2 p^{1+2 \sigma}} \\ 
&+\frac{(2+i u)(i u / 2) \cos \left(2 t_1 \log p\right)+(2+i v)(i v / 2) \cos \left(2\left(t_1+t_2\right) \log p\right)}{2 p^{1+2 \sigma}} \\
&+\frac{(2+i u)(2+i v) \cos \left(\left(2 t_1+t_2\right) \log p\right)}{2 p^{1+2 \sigma}}+D_p(u, v),
\end{align*}
where $D_p(u, v)$ satisfies $\left|D_p(u, v)\right| \ll \frac{1+|u|^3+|v|^3}{p^{3 / 2+\sigma}} \ll \frac{1+|u|^3+|v|^3}{p^{3 / 2}}$, and its partial derivatives satisfy $\left|\frac{\partial D_p(u, v)}{\partial u}\right| \ll \frac{1+u^2+v^2}{p^{3 / 2}},\left|\frac{\partial D_p(u, v)}{\partial v}\right| \ll \frac{1+u^2+v^2}{p^{3 / 2}}$ and $\left|\frac{\partial D_p(u, v)}{\partial u \partial v}\right| \ll \frac{1+|u|+|v|}{p^{3 / 2}}$. In the last step, we have used that $\mathbb{E} \sum_{j=3}^{\infty} \frac{\left((2+i u) R_p\left(t_1\right)+(2+i v) R_p\left(t_1+t_2\right)\right)^j}{j!} = O(\sum_{j=3}^{\infty} \frac{\left(4+|u|+|v|\right)^j3^j}{j!(p^{1/2+\sigma}-3)^j} )$.

The conclusion of Lemma \ref{lemma:mean-square-randomEulerproducts-lowerbd} now follows as in the proof of Lemma \ref{lemma:mean-square-randomEulerproducts-upperbd}, using the independence of $\X$ at distinct primes and the estimate $\sum_{p> x} 1 / p^{3 / 2} \ll 1 /(\sqrt{x} \log x)$.
\end{proof}

\begin{lemma}\label{lemma:mean-square-randomEulerproducts-high moment}
    Let $\X$ be a $\GL_2$ random multiplicative function. For any real $\alpha, \beta \geq 0$, any real $100(1 + \max\{\alpha^2 , \beta^2\}) \leq x \leq y$, and any real $\sigma \geq - 1/\log y$ and $t_1 , t_2$, we have
\begin{eqnarray}
&& \E \prod_{x < p \leq y} \left|1 -\frac{\X(p)}{p^{1/2+\sigma + it_1}}+\frac{1}{p^{1+2\sigma + 2it_1}}\right|^{-2\alpha} \left|1 -\frac{\X(p)}{p^{1/2+\sigma + it_2}}+\frac{1}{p^{1+2\sigma + 2it_2}}\right|^{-2\beta} \nonumber \\
& = & \exp\{\sum_{x < p \leq y} \frac{\alpha^2 + \beta^2 + (\alpha^2 - \alpha)\cos(2t_1 \log p) + (\beta^2 - \beta)\cos(2t_2 \log p)}{p^{1 + 2\sigma}} + \nonumber \\
&& + \sum_{x < p \leq y} \frac{2\alpha\beta(\cos((t_1 + t_2)\log p) + \cos((t_1 - t_2)\log p))}{p^{1 + 2\sigma}} + O(\frac{\max\{\alpha, \beta, \alpha^3 , \beta^3\}}{\sqrt{x} \log x}) \} . \nonumber
\end{eqnarray}
\end{lemma}

\begin{proof}
This is a $\GL_2$ analogue of \cite[Euler Product Results 1,~2]{H19}.
As in Lemma~\ref{lemma:mean-square-randomEulerproducts-upperbd}, set
\[
R_p(t):=-\Re\log\left(
1-\frac{\X(p)}{p^{1/2+\sigma+it}}
+\frac{1}{p^{1+2\sigma+2it}}
\right).
\]
Thus, we may write
\begin{align*}
&\left|1-\frac{\X(p)}{p^{1/2+\sigma+it_1}}
+\frac{1}{p^{1+2\sigma+2it_1}}\right|^{-2\alpha}
\left|1-\frac{\X(p)}{p^{1/2+\sigma+it_2}}
+\frac{1}{p^{1+2\sigma+2it_2}}\right|^{-2\beta}\\
&=\exp\left\{2\alpha R_p(t_1)+2\beta R_p(t_2)\right\}\\
&=1+\sum_{j=1}^{\infty}
\frac{\left(2\alpha R_p(t_1)+2\beta R_p(t_2)\right)^j}{j!}.
\end{align*}

For primes $x<p\leq y$, our assumptions give
\[
\frac{1}{p^{1/2+\sigma}}\leq\frac{e}{\sqrt p}<\frac{e}{10},
\qquad
\frac{\alpha+\beta}{p^{1/2+\sigma}}
\leq
\frac{e(\alpha+\beta)}
{10\sqrt{1+\max\{\alpha^2,\beta^2\}}}
\leq\frac{e}{5}.
\]
As in
Lemma~\ref{lemma:mean-square-randomEulerproducts-upperbd}, we have
\[
R_p(t)
=
\frac{\X(p)\cos(t\log p)}{p^{1/2+\sigma}}
+\frac{(\X(p)^2-2)\cos(2t\log p)}{2p^{1+2\sigma}}
+O\left(\frac{1}{p^{3/2+3\sigma}}\right).
\]
In particular, $|R_p(t)|\ll p^{-1/2-\sigma}$, and hence
$|2\alpha R_p(t_1)+2\beta R_p(t_2)|\ll 1$.
It follows that
\[
\left|
\E\sum_{j=3}^{\infty}
\frac{\left(2\alpha R_p(t_1)+2\beta R_p(t_2)\right)^j}{j!}
\right|
\ll
\frac{(\alpha+\beta)^3}{p^{3/2+3\sigma}}.
\]

Using $\E\X(p)=0$ and $\E\X(p)^2=1$, we obtain
\begin{align*}
&\ \ \ \ \E
\left|1-\frac{\X(p)}{p^{1/2+\sigma+it_1}}
+\frac{1}{p^{1+2\sigma+2it_1}}\right|^{-2\alpha}
\left|1-\frac{\X(p)}{p^{1/2+\sigma+it_2}}
+\frac{1}{p^{1+2\sigma+2it_2}}\right|^{-2\beta}\\
&=1
-\frac{\alpha\cos(2t_1\log p)+\beta\cos(2t_2\log p)}
{p^{1+2\sigma}}\\
&\quad
+\frac{
2(\alpha\cos(t_1\log p)
+\beta\cos(t_2\log p))^2
}
{p^{1+2\sigma}}
+O\left(
\frac{\max\{\alpha,\beta,\alpha^3,\beta^3\}}
{p^{3/2+3\sigma}}
\right)\\
&=1+
\frac{
\alpha^2+\beta^2
+(\alpha^2-\alpha)\cos(2t_1\log p)
+(\beta^2-\beta)\cos(2t_2\log p)}
{p^{1+2\sigma}}\\
&\quad+
\frac{
2\alpha\beta\bigl(
\cos((t_1+t_2)\log p)+\cos((t_1-t_2)\log p)
\bigr)}
{p^{1+2\sigma}}+
O\left(
\frac{\max\{\alpha,\beta,\alpha^3,\beta^3\}}
{p^{3/2+3\sigma}}
\right)\\
&=\exp \{
\frac{
\alpha^2+\beta^2
+(\alpha^2-\alpha)\cos(2t_1\log p)
+(\beta^2-\beta)\cos(2t_2\log p)}
{p^{1+2\sigma}}\\
&\quad+
\frac{
2\alpha\beta\bigl(
\cos((t_1+t_2)\log p)+\cos((t_1-t_2)\log p)
\bigr)}
{p^{1+2\sigma}}+
O\left(
\frac{\max\{\alpha,\beta,\alpha^3,\beta^3\}}
{p^{3/2+3\sigma}}
\right)\}.
\end{align*}
The conclusion now follows by independence  of $\X$  at distinct primes and
the estimate $\sum_{p> x} 1 / p^{3 / 2} \ll 1 /(\sqrt{x} \log x)$.
\end{proof}

\smallskip

\subsection{Moments estimates for probabilistic models}\label{subsecmeans}
We will require
the following two lemmas concerning estimates for even moments of quantities involving the $\GL_2$ random multiplicative function. 

\begin{lemma}\label{lemma:fourth-moment-random-model}
Let \(\mathbb X\) be a $\GL_2$ random multiplicative function. Let $(c(n))_{n \leq x}$ be any sequence of complex numbers such that $|c(n)|\leq 1$. Then, for every $x\geq 1$, we have
\[
\mathbb E\left|
    \sum_{n\leq x}c(n)\mathbb X(n)
\right|^{4}
\ll x^{2}(\log 2x)^{4}.
\]
\end{lemma}

\begin{proof}
Let
\[
S(x):=\sum_{n\leq x}c(n)\mathbb X(n).
\]
By the Hecke relation,
\[
S(x)^{2}
=
\mathop{\sum\sum}_{m,n\leq x}c(m)c(n)
\sum_{\substack{d\mid(m,n)}}
\mathbb X\left(\frac{mn}{d^{2}}\right).
\]
For \(r\geq 1\), define
\[
A_x(r)
:=\mathop{\sum\sum}_{m,n\leq x}c(m)c(n)
\sum_{\substack{d\mid(m,n)\\mn/d^{2}=r}}1.
\]
Then
\[
S(x)^{2}
=
\sum_{r\leq x^{2}}A_x(r)\mathbb X(r).
\]
The orthogonality relation \eqref{eq:orthoganol-relation-GL2-RMf} yields
\[
\begin{aligned}
\mathbb E|S(x)|^{4}
&=
\mathbb E\left|
    \sum_{r\leq x^{2}}A_x(r)\mathbb X(r)
\right|^{2} =
\sum_{r\leq x^{2}}\left|A_x(r)\right|^{2}.
\end{aligned}
\]

To evaluate \(A_x(r)\), write
\(
m=da,n=db.
\)
The condition \(mn/d^{2}=r\) becomes \(ab=r\), and the inequalities
\(m,n\leq x\) are equivalent to
\[
d\leq \frac{x}{\max\{a,b\}}.
\]
Hence
\[
\left|A_x(r)\right|
\leq
\mathop{\sum}_{\substack{a,b\geq 1\\ab=r}}
\left\lfloor
    \frac{x}{\max\{a,b\}}
\right\rfloor\leq
\frac{x\tau(r)}{\sqrt r}.
\]
It follows that
\[
\mathbb E|S(x)|^{4}
\leq
x^{2}\sum_{r\leq x^{2}}\frac{\tau(r)^{2}}{r}\ll x^{2}(\log 2x)^{4}.
\]
The proof of Lemma \ref{lemma:fourth-moment-random-model} is finished.
\end{proof}
\begin{lemma}\label{lemma:2l-th-moment-prime-sums}
Let \(\X\) be a $\GL_2$ random multiplicative function, and let
\(a(p),a(p^2)\) be any complex number with
\(
    |a(p)|\leq 1,
    |a(p^2)|\leq 1.
\) Let $P\geq 1$ be sufficiently
large.
For any integer \(l\geq \log\log P\),
\[
\mathbb{E}
\left|
    \sum_{p\leq P}
    \left(
        \frac{a(p)\X(p)}{\sqrt p}
        +
        \frac{a(p^2)(\X(p^2)-1)}{p}
    \right)
\right|^{2l}
\ll
l!\bigl(80\log\log P\bigr)^l.
\]
\end{lemma}
\begin{proof}

    Let
\[
    \W_p
    :=
    \frac{a(p)\X(p)}{\sqrt{p}}
  ,~
    S_P:=\sum_{p\leq P}\W_p,~T_P:=\sum_{p\leq P}  \frac{a(p^2)(\X(p^2)-1)}{p}
\]
Hence, we need to show
\[
\mathbb{E}
\left|
     S_P+T_P
\right|^{2l}
\ll
l!\bigl(80\log\log P\bigr)^l.
\]
By Jensen's inequality, we get 
\[
\mathbb{E}
\left|
     S_P+T_P
\right|^{2l}
\leq
2^{2l-1}(\mathbb{E}\left|
     S_P
\right|^{2l}+\mathbb{E}\left|
     T_P
\right|^{2l}).
\]
By definition of \(\X\), we see that
\(
    |\X(p)|\leq 2,~
    |\X(p^2)-1|
    =
    |4\cos^2\theta_p-2|
    \leq 2.
\)
 Using Mertens' estimate together with $(l)! \geq (l/e)^{l}$, we see for \(l\geq \log\log P\) that 
\[\mathbb{E}\left|
     T_P
\right|^{2l}\leq (2.1\log \log P)^{2l}\leq l!\bigl(20\log\log P\bigr)^l\]

By definition, the random
variables \((\W_p)_{p~ \rm prime}\) are independent.
It follows from \eqref{eq:orthoganol-relation-GL2-RMf} that
\(
    \mathbb{E}\W_p=0.
\)

Consequently,
\[
    |\W_p|
    \leq
    \frac{2}{\sqrt{p}}
    =:B_p.
\]
Set
\(
    V_P:=\sum_{p\leq P}B_p^2.
\)
Note that
\[
\begin{aligned}
V_P
=
4\sum_{p\leq P}\frac{1}{p}
=
4\log\log P+O(1)\leq
5\log\log P
\end{aligned}
\]
For \(j=1,2\), let
\(
    \U_{p,1}:=\Re \W_p,~
    \U_{p,2}:=\Im \W_p.
\)
The random variables \(\U_{p,1}\) are independent and satisfy
\[
    \mathbb{E}\U_{p,1}=0 \quad\text{and}\quad -B_p\leq \U_{p,1}\leq B_p.
\]
Applying Hoeffding's inequality (see Gut
\cite[Chapter~3, Theorem~1.3]{G13}) yields
\[
\mathbb{P}
\left(
    \left|
        \Re S_P
    \right|>u
\right)=
\mathbb{P}
\left(
    \left|
        \sum_{p\leq P}\U_{p,1}
    \right|>u
\right)
\leq
2\exp\left(-\frac{u^2}{2V_P}\right).
\]
Similarly, we have \[
\mathbb{P}
\left(
    \left|
        \Im S_P
    \right|>u
\right)=
\mathbb{P}
\left(
    \left|
        \sum_{p\leq P}\U_{p,2}
    \right|>u
\right)
\leq
2\exp\left(-\frac{u^2}{2V_P}\right).
\]
Since $ |S_P|>u$ implies
\[
    |\Re S_P|>\frac{u}{\sqrt 2}
    \quad\text{or}\quad
    |\Im S_P|>\frac{u}{\sqrt 2},
\]
we see that
\[
    \mathbb{P}(|S_P|>u)
    \leq
    4\exp\left(-\frac{u^2}{4V_P}\right).
\]
Thus, we have
\begin{align*}
\mathbb{E}|S_P|^{2l}
&=
2l\int_0^\infty
u^{2l-1}\mathbb{P}(|S_P|>u)\,du
\\
&\leq
8l\int_0^\infty
u^{2l-1}
\exp\left(-\frac{u^2}{4V_P}\right)\,du
\\
&\leq 4
l!\bigl(20\log\log P\bigr)^l.
\end{align*}
It follows that
\[
\begin{aligned}
\mathbb{E}|S_P+T_P|^{2l}
&\leq
2^{2l-1}
\left(
4l!(20\log\log P)^l+l!(20\log\log P)^l
\right)\ll
\,l!(80\log\log P)^l.
\end{aligned}
\]
This proves the lemma.
\end{proof}

\section{Mean value estimates for Hecke eigenvalues}

The following lemma provides an even moment estimate for sums of Hecke eigenvalues of holomorphic Hecke cusp form. This is a $\GL_2$ analogue of \cite[Lemma 1]{H23}, which concerns an even moment estimate for character sums.
\begin{lemma}[Even moment estimate for holomorphic Hecke cusp form 1]\label{lemma: evenmomentlem-modular-1}
Let $x\geq1$, and let $P$ be a
large real number. Let $(c(n))_{n \leq x}$ be any complex numbers such that $|c(n)|\leqslant 1$. 
 Let $Q(\lambda_f) :=     \sum_{p\le P}\left(
\frac{a(p)\lambda_f(p)}{\sqrt p}
+
\frac{a(p^2)(\lambda_f(p^2)-1)}{p}    \right)$, where the $a(p)$ and  $a(p^2)$ are any complex numbers such that $|a(p)|\leqslant 1$ and  $|a(p^2)|\leqslant 1$. Let \(l\) be an integer such that \(l\geq \log\log P\).

Then for $ xP^{2l} < k/100$,  we have
$$ \E^{\rm holo} \Biggl|\sum_{n \leq x} c(n) \lambda_f(n) \Biggr|^{2} |Q(\lambda_f)|^{2l} \ll x\log (2x)(l!)\left(320\log \log P\right)^{l} $$
\end{lemma}

Using a different method, we are able to obtain another upper bound.

\begin{lemma}[Even moment estimate for holomorphic Hecke cusp form 2]\label{lemma: evenmomentlem-modular-2}
Let $P$ be a
large real number, and let $(c(n))_{n \leq x}$ be any complex numbers such that $|c(n)|\leqslant 1$. 
%Let $\mathcal{P}$ be any finite set of primes, let $\mathcal{Q}$ be any (non-empty) set consisting of some elements of $\mathcal{P}$ and squares of elements of $\mathcal{P}$, and write $U := \max\{q \in \mathcal{Q}\}$.
 Let $Q(\lambda_f) := \sum_{p\le P}
\left(
\frac{a(p)\lambda_f(p)}{\sqrt p}
+
\frac{a(p^2)(\lambda_f(p^2)-1)}{p}    \right)$, where the $a(p)$ and  $a(p^2)$ are any complex numbers such that $|a(p)|\leqslant 1$ and  $|a(p^2)|\leqslant 1$.  Let $l$ be a natural number.

Then for $1\leq x < k/100$, we have
$$ \E^{\rm holo} \Biggl|\sum_{n \leq x} c(n) \lambda_f(n) \Biggr|^{2} |Q(\lambda_f)|^{2l} \ll 10^{2l}x\left(\frac{\sqrt P}{\log P}\right)^{2l} $$
\end{lemma}

Using arguments similar to the proof of Lemma \ref{lemma: evenmomentlem-modular-1}, we obtain analogous estimates for Hecke eigenvalues of the Maass Hecke cusp form.
\begin{lemma}[Even moment estimate for Hecke-Maass cusp form]\label{lemma：evenmomentlem-maass}
Let $x\geq1$, and let $P$ be a
large real number. Let $(c(n))_{n \leq x}$ be any complex numbers such that $|c(n)|\leqslant 1$. 
 Let $Q(\lambda_j) := \sum_{p\le P}
\left(
\frac{a(p)\lambda_j(p)}{\sqrt p}
+
\frac{a(p^2)(\lambda_j(p^2)-1)}{p}    \right),$ where the $a(p)$ and  $a(p^2)$ are any complex numbers such that $|a(p)|\leqslant 1$ and  $|a(p^2)|\leqslant 1$.
Let \(l\) be an integer such that \(l\geq \log\log P\)

Assume that $xP^l\leq T^{1-\eta}$ for any fixed small $\eta$. Then we have
$$ \E^{\rm Maass} \Biggl|\sum_{n \leq x} c(n) \lambda_j(n) \Biggr|^{2} |Q(\lambda_j)|^{2l} \ll x\log (2x)l!\left(320\log \log P\right)^{l} $$
\end{lemma}

\begin{proof}[Proof of Lemma \ref{lemma: evenmomentlem-modular-1}]
Since $xP^{2l} < k/100$, we may apply Lemma \ref{lemma: Eigenvalue like random model-holomorphic} to obtain
\begin{align*}
&\ \ \ \ \mathbb{E}^{\rm Holo}
\left|
    \sum_{n\leq x} c(n)\lambda_f(n)
\right|^2
\left|
    \sum_{p\leq P}
    \frac{a(p)\lambda_f(p)}{\sqrt{p}}
    +
    \frac{a(p^2)(\lambda_f(p^2)-1)}{p}
\right|^{2l}
\\&=
\mathbb{E}
\left|
    \sum_{n\leq x} c(n)\X(n)
\right|^2
\left|
    \sum_{p\leq P}
    \frac{a(p)\X(p)}{\sqrt{p}}
    +
    \frac{a(p^2)(\X(p^2)-1)}{p}
\right|^{2l}
+
O\left(
    x^{2+\varepsilon}
    (10P^{1/2})^{2l}e^{-k}
\right).
\end{align*}
The error term here is acceptably small.
Using H\"{o}lder's inequality, we have
\begin{align*}
&\ \ \ \ \mathbb{E}
\left|
    \sum_{n\leq x} c(n)\X(n)
\right|^2
\left|
    \sum_{p\leq P}
    \frac{a(p)\X(p)}{\sqrt{p}}
    +
    \frac{a(p^2)(\X(p^2)-1)}{p}
\right|^{2l}
\\&
\leq
\left(
    \mathbb{E}
    \left|
        \sum_{n\leq x}c(n)\X(n)
    \right|^2
\right)^{1/2}
\left(
    \mathbb{E}
    \left|
        \sum_{n\leq x}c(n)\X(n)
    \right|^4
\right)^{1/4}
\left(
    \mathbb{E}
    \left|
        \sum_{p\leq P}
        \frac{a(p)\X(p)}{\sqrt{p}}
        +
        \frac{a(p^2)(\X(p^2)-1)}{p}
    \right|^{8l}
\right)^{1/4}.
\end{align*}
It follows from  \eqref{eq:orthoganol-relation-GL2-RMf} that
\[
\mathbb{E}
\left|
    \sum_{n\leq x}c(n)\X(n)
\right|^2
\ll x.
\]
Since \(l\geq \log\log P\), we apply  Lemma \ref{lemma:fourth-moment-random-model} and Lemma \ref{lemma:2l-th-moment-prime-sums} to obtain  \begin{align*}
&\ \ \ \ \mathbb{E}
\left|
    \sum_{n\leq x} c(n)\X(n)
\right|^2
\left|
    \sum_{p\leq P}
    \frac{a(p)\X(p)}{\sqrt{p}}
    +
    \frac{a(p^2)(\X(p^2)-1)}{p}
\right|^{2l}
\\&
\leq
   x^{1/2}
\left(
   x^2(\log (2x))^4
\right)^{1/4}
\left(
   (4l)!\bigl(80\log\log P\bigr)^{4l}
\right)^{1/4}\leq x\log (2x) (l!)\bigl(320\log\log P\bigr)^{l}.
\end{align*}

\end{proof}

\begin{proof}[Proof of Lemma \ref{lemma: evenmomentlem-modular-2}]
By  Deligne’s work \cite{D74}, we have
\[
|\lambda_f(p)|\le 2,\quad 
|\lambda_f(p^2)|\le 3 ,
\]
which gives
\[
|Q(\lambda_f)|^{2l}
=
\left|
\sum_{p\le P}
\frac{a(p)\lambda_f(p)}{\sqrt p}
+
\frac{a(p^2)(\lambda_f(p^2)-1)}{\sqrt{p^2}}
\right|^{2l}
\le
\left|
\sum_{p\le P}
\frac{2}{\sqrt p}
+
\frac{3}{p}
\right|^{2l}
\le
10^{2l}\left(\frac{\sqrt P}{\log P}\right)^{2l}.
\]
Hence  
\begin{equation}
\begin{aligned}
&~\mathbb{E}^{\mathrm{holo}}
\left|
\sum_{n\le x} c(n)\lambda_f(n)
\right|^2
|Q(\lambda_f)|^{2l}\\
\le&~
10^{2l}\left(\frac{\sqrt P}{\log P}\right)^{2l}
\sum_{n_1,n_2\le x}
c(n_1)\overline{c(n_2)}
\mathbb{E}^{\mathrm{Hecke}}
\lambda_f(n_1)\lambda_f(n_2)
\end{aligned}
\end{equation}
By lemma \ref{lemma:orthogonality-holomorphic}, this is 
\begin{equation}
\begin{aligned}
\ll&~
10^{2l}\left(\frac{\sqrt P}{\log P}\right)^{2l}
\left(
\sum_{n\le x}|c(n)|^2
+
O(x^2e^{-k})
\right)\\
\ll 
&~10^{2l}x\left(\frac{\sqrt P}{\log P}\right)^{2l}
\end{aligned}
\end{equation}
This completes the
proof of the lemma. 
\end{proof}

\begin{proof}[Proof of Lemma \ref{lemma：evenmomentlem-maass}]
We may apply Lemma \ref{lemma: Eigenvalue like random model-maass} to obtain
\begin{align*}
&\ \ \ \ \mathbb{E}^{\rm Maass}
\left|
    \sum_{n\leq x} c(n)\lambda_j(n)
\right|^2
\left|
    \sum_{p\leq P}
    \frac{a(p)\lambda_j(p)}{\sqrt{p}}
    +
    \frac{a(p^2)(\lambda_j(p^2)-1)}{p}
\right|^{2l}
\\&=
\mathbb{E}
\left|
    \sum_{n\leq x} c(n)\X(n)
\right|^2
\left|
    \sum_{p\leq P}
    \frac{a(p)\X(p)}{\sqrt{p}}
    +
    \frac{a(p^2)(\X(p^2)-1)}{p}
\right|^{2l}\\&
+
O\left(
    x^{2+\varepsilon}    (10P^{1/2+\varepsilon})^{2l}T^{-(1-\varepsilon)}
\right)+O\left(
    x^{5/2+\varepsilon}(10P^{3/4+\varepsilon})^{2l}T^{-2}
\right).
\end{align*}
The error term here is acceptable provided that $xP^l\leq T^{1-\eta}$.
Using H\"{o}lder's inequality, we have
\begin{align*}
&\ \ \ \ \mathbb{E}
\left|
    \sum_{n\leq x} c(n)\X(n)
\right|^2
\left|
    \sum_{p\leq P}
    \frac{a(p)\X(p)}{\sqrt{p}}
    +
    \frac{a(p^2)(\X(p^2)-1)}{p}
\right|^{2l}
\\&
\leq
\left(
    \mathbb{E}
    \left|
        \sum_{n\leq x}c(n)\X(n)
    \right|^2
\right)^{1/2}
\left(
    \mathbb{E}
    \left|
        \sum_{n\leq x}c(n)\X(n)
    \right|^4
\right)^{1/4}
\left(
    \mathbb{E}
    \left|
        \sum_{p\leq P}
        \frac{a(p)\X(p)}{\sqrt{p}}
        +
        \frac{a(p^2)(\X(p^2)-1)}{p}
    \right|^{8l}
\right)^{1/4}.
\end{align*}
It follows from  \eqref{eq:orthoganol-relation-GL2-RMf} that
\[
\mathbb{E}
\left|
    \sum_{n\leq x}c(n)\X(n)
\right|^2
\ll x.
\]
Since \(l\geq \log\log P\), we apply  Lemma \ref{lemma:fourth-moment-random-model} and Lemma \ref{lemma:2l-th-moment-prime-sums} to obtain  \begin{align*}
&\ \ \ \ \mathbb{E}
\left|
    \sum_{n\leq x} c(n)\X(n)
\right|^2
\left|
    \sum_{p\leq P}
    \frac{a(p)\X(p)}{\sqrt{p}}
    +
    \frac{a(p^2)(\X(p^2)-1)}{p}
\right|^{2l}
\\&
\leq
   x^{1/2}
\left(
   x^2(\log (2x))^4
\right)^{1/4}
\left(
   (4l)!\bigl(80\log\log P\bigr)^{4l}
\right)^{1/4}\leq x\log (2x) (l!)\bigl(320\log\log P\bigr)^{l}.
\end{align*}

\end{proof}

Using Lemma \ref{lemma: evenmomentlem-modular-1} and Lemma \ref{lemma: evenmomentlem-modular-2}, we obtain the following key proposition using arguments as in the proof of \cite[Proposition 1]{H23}. 

\begin{proposition}
[Holomorphic Hecke Eigenvalues behave like random model]\label{prop:holomorphic-Eigenvalues-behave like-random-model}
Let functions $g_j$ ($-N\leq j\leq N+1$)  with associated parameters $N$ and $\delta$  be as  in Lemma \ref{lemma:partition-of-unity}. Let $(c(n))_{n \leq x}$ be any sequence of complex numbers such that $|c(n)|\leq 1$. Let 
$ (j(i))_{i \leq Y} $
be any indices satisfying $-N \leq j(i)\leq N+1$ for $i\leq Y$, and let $(a_{i}(p^b))_{p \leq P}$ for $i\leq Y,~b=1,2$ be complex numbers such that $|a_{i}(p^b)|\leq 1$.  Suppose that $x \geq 1$. Let $P$ be large, and let $Y \in \N$.  Let $\X(n)$ denote a $\GL_2$ random multiplicative function. 
\begin{itemize}
    \item For 
$x P^{32000(Y/\delta)^2 \log(N\log P)} < k/100$, we have
\begin{align}
\label{eq:holomorphic-Eigenvalues-behave like-random-model1}
&\ \ \  \E^{holo} \prod_{i=1}^{Y} g_{j(i)}(\Re(\sum_{p \leq P} \frac{a_{i}(p) \lambda_f(p)}{\sqrt{p}} + \frac{a_{i}(p^2) (\lambda_f(p^2)-1)}{p})) \Biggl|\sum_{n \leq x} c(n) \lambda_f(n) \Biggr|^2 \\
& =  \E \prod_{i=1}^{Y} g_{j(i)}(\Re(\sum_{p \leq P} \frac{a_{i}(p) \X(p)}{\sqrt{p}} + \frac{a_{i}(p^2) (\X(p^2)-1)}{p})) \Biggl|\sum_{n \leq x} c(n) \X(n) \Biggr|^2 + O\left(\frac{x\log (2x)}{(N \log P)^{2000Y(1/\delta)^2}} \right) . \nonumber
\end{align}
    \item 
For $x P^{1000Y(1/\delta)^2 \sqrt{P}} < k/100$ , $N\leq 10\log P$ and $Y\leq (\log P)^{100}$ we have
\begin{align}
\label{eq:holomorphic-Eigenvalues-behave like-random-model2}
&\ \ \  \E^{holo} \prod_{i=1}^{Y} g_{j(i)}(\Re(\sum_{p \leq P} \frac{a_{i}(p) \lambda_f(p)}{\sqrt{p}} + \frac{a_{i}(p^2) (\lambda_f(p^2)-1)}{p})) \Biggl|\sum_{n \leq x} c(n) \lambda_f(n) \Biggr|^2 \\
& =  \E \prod_{i=1}^{Y} g_{j(i)}(\Re(\sum_{p \leq P} \frac{a_{i}(p) \X(p)}{\sqrt{p}} + \frac{a_{i}(p^2) (\X(p^2)-1)}{p})) \Biggl|\sum_{n \leq x} c(n) \X(n) \Biggr|^2 + O\left(\frac{x}{(N \log P)^{Y(1/\delta)^2}} \right) . \nonumber
\end{align}
\end{itemize}
 
\end{proposition}
\begin{remark}
Lemma \ref{lemma: evenmomentlem-modular-1}
is applied in establishing
\eqref{eq:holomorphic-Eigenvalues-behave like-random-model1} in
Proposition~\ref{prop:holomorphic-Eigenvalues-behave like-random-model},
which in turn is used in the proof of
Theorem~\ref{mainthm:low-moment-modular-sharp}.
    The presence of the factor \(\log x\) in the bound given by Lemma~\ref{lemma: evenmomentlem-modular-1}, introduces an additional factor of \(\log x\) into the error term of \eqref{eq:holomorphic-Eigenvalues-behave like-random-model1} in
Proposition~\ref{prop:holomorphic-Eigenvalues-behave like-random-model}, compared with the corresponding error term in \cite[Proposition~1]{H23}, a loss which leads us to impose the stronger restriction\footnote{The weaker condition $x\leq k$ is already required when we apply the Petersson trace formula in the proof.} $
x\leq k\exp\left(-(\log\log k)^2\right)
$  in the proof of Theorem~\ref{mainthm:low-moment-modular-sharp}.
\end{remark}
\begin{remark}\eqref{eq:holomorphic-Eigenvalues-behave like-random-model2} in
Proposition~\ref{prop:holomorphic-Eigenvalues-behave like-random-model} is used  in the proof of
Theorem~\ref{mainthm:low-moment-modular}. The error term of \eqref{eq:holomorphic-Eigenvalues-behave like-random-model2} in
Proposition~\ref{prop:holomorphic-Eigenvalues-behave like-random-model} is of the same order as that in \cite[Proposition~1]{H23}.
    This estimate, however, requires a stronger restriction on $P$, and thus we can not  get a sufficient saving in the final upper bound of Theorem \ref{mainthm:low-moment-modular} (our final saving will be of the form $(\sqrt{\log \log P})^q$).
\end{remark}
\begin{proof}[Proof of Proposition \ref{prop:holomorphic-Eigenvalues-behave like-random-model}]
Using Taylor expansion, we can write $g_{j}(x) = \tilde{g}_{j}(x) + r_{j}(x)$, where $\tilde{g}_{j}(x)=
\sum_{k=0}^{2S-1}
\frac{g_{j}^{(k)}(0)}{k!}x^k
$ is a polynomial of degree $2S-1$. We infer from property (iv) of Lemma  \ref{lemma:partition-of-unity} that
\begin{align}\label{eq:bound-for-remainder}
|r_{j}(x)| \leq \frac{|x|^{2S}}{(2S)!} \sup_{|y| \leq |x|} |\frac{d^{2S}}{dy^{2S}} g_{j}(y)| \ll \frac{N |2\pi x/\delta|^{2S}}{\delta S (2S)!}.
\end{align} Here the factor $N$ arises from the upper bound for derivative of  $g_{N+1}(x)=1-\sum_{|j| \leq N} g_{j}(x)$. Provided that $x^2 P^{4SY} < k^2/10^4$,  we expand all the polynomials and the square out and apply Lemma \ref{lemma: Eigenvalue like random model-holomorphic} to find that
\begin{align}\label{eq:holomorphic-heckeeigenvalue-closeto-random-model-ploynomialcase}
&\ \ \  \E^{\rm holo} \prod_{i=1}^{Y} \tilde{g}_{j(i)}(\Re(\sum_{p \leq P} \frac{a_{i}(p) \lambda_f(p)}{\sqrt{p}} + \frac{a_{i}(p^2) (\lambda_f(p^2)-1)}{p})) \Biggl|\sum_{n \leq x} c(n) \lambda_f(n) \Biggr|^2  \\
& =  \E \prod_{i=1}^{Y} \tilde{g}_{j(i)}(\Re( \sum_{p \leq P} \frac{a_{i}(p) \X(p)}{\sqrt{p}} + \frac{a_{i}(p^2) (\X(p^2)-1)}{p})) \Biggl|\sum_{n \leq x} c(n) \X(n) \Biggr|^2\nonumber \\ & \ \ \  +  O\left(
x^{2+\epsilon}\left(\frac{7N}{\delta}e^{64\pi/\delta}P^{S}\right)^Y
e^{-k}\right). \nonumber
\end{align}
Here we use the fact that  $|g_{j}^{(k)}(x)|\leq \frac{6N}{\delta(k+1)} (\frac{2\pi}{\delta})^k$ for $-N \leq j\leq N+1$.

Next, dividing up according to the smallest index $i$ at which we get a remainder, we see the contribution from all of the remainders $r_{j(i)}(\cdot)$ to the left hand side of \eqref{eq:holomorphic-Eigenvalues-behave like-random-model1} is
\begin{align}\label{eq:contribution-from-the-remainders}
&\ll \E^{\rm holo} \sum_{i=1}^{Y} \frac{N |2\pi/\delta|^{2S}}{\delta S (2S)!} |\sum_{p \leq P} \frac{a_{i}(p) \lambda_f(p)}{\sqrt{p}} + \frac{a_{i}(p^2) (\lambda_f(p^2)-1)}{p}|^{2S} \cdot  \\
&\ \ \   \cdot \prod_{l=1}^{i-1} \Biggl(1 + O(\frac{N |2\pi/\delta|^{2S}}{\delta S (2S)!} |\sum_{p \leq P} \frac{a_{l}(p) \lambda_f(p)}{\sqrt{p}} + \frac{a_{l}(p^2) (\lambda_f(p^2)-1)}{p}|^{2S})\Biggr) \Biggl|\sum_{n \leq x} c(n) \lambda_f(n) \Biggr|^2, \nonumber
\end{align}
 where we have used the arithmetic-geometric mean inequality.
Here we used the fact that $|\tilde{g}_{j(l)}(x)| \leq |g_{j(l)}(x)| + |r_{j(l)}(x)| \leq 1 + O(\frac{N |2\pi x/\delta|^{2S}}{\delta S (2S)!})$. 
Using Lemma \ref{lemma: evenmomentlem-modular-1} provided that $xP^{2SY}\leq k/100$ and $S\geq\log\log P$, along with  $(2S)! \geq (2S/e)^{2S}$, we find this is all
\begin{eqnarray}
& \ll & x\log (2x) \cdot \sum_{i=1}^{Y} \frac{N|2\pi/\delta|^{2S}}{\delta S (2S)!} \Biggl( \sqrt{iS} (\frac{iS}{e})^S (320\log\log P)^{S} \Biggr) \cdot \nonumber \\
&& \cdot \prod_{l=1}^{i-1} \Biggl(1 + O\Biggl(\frac{N|2\pi/\delta|^{2S}}{\delta S (2S)!} (\frac{iS}{e})^S (320\log\log P)^{S} \Biggr) \Biggr) \nonumber \\
& \ll & x\log (2x) \cdot \sum_{i=1}^{Y} \frac{N}{\delta} \sqrt{\frac{i}{S}} (\frac{e (\pi/\delta)^2 i (320\log\log P)}{S})^S \cdot \nonumber \\
&& \cdot \Biggl( 1 + O\Biggl(\frac{N}{\delta S} (\frac{e (\pi/\delta)^2 i (320\log\log P)}{S})^S \Biggr) \Biggr)^{i-1} . \nonumber
\end{eqnarray}
 One has the same overall bound for the contribution from the remainders $r_{j(i)}(\cdot)$ to the right hand side of \eqref{eq:holomorphic-Eigenvalues-behave like-random-model1}, since one has the same bound for $\E |\sum_{n \leq x} c(n) \X(n) |^{2} |Q(\X)|^{2k}$ as for the Hecke eigenvalue averages in Lemma \ref{lemma: evenmomentlem-modular-1}.

Now  we set $S = 16000Y \lfloor (1/\delta)^2 \log(N\log P) \rfloor$. Then the condition $xP^{2SY}\leq k/100$  is satisfied in view of our assumption that $x P^{32000(Y/\delta)^2 \log(N\log P)} < k/100$. The error term produced by all of the remainders is
$$ \ll x\log (2x) \cdot \sum_{i=1}^{Y} \frac{N}{\delta} \sqrt{\frac{i}{S}} 0.6^S \left( 1 + O(\frac{N}{\delta S} 0.6^S ) \right)^{i-1} \ll x\log (2x) \cdot NY \cdot 0.6^S \ll \frac{x\log (2x)}{(N \log P)^{2000Y(1/\delta)^2}} , $$
as desired. Furthermore, with  $x P^{32000(Y/\delta)^2 \log(N\log P)} <  k/100$, the error term in \eqref{eq:holomorphic-heckeeigenvalue-closeto-random-model-ploynomialcase}  is smaller than that stated in \eqref{eq:holomorphic-Eigenvalues-behave like-random-model1}.
The proof of  \eqref{eq:holomorphic-Eigenvalues-behave like-random-model1} is now finished.

Next, we turn to the proof of \eqref{eq:holomorphic-Eigenvalues-behave like-random-model2}. We begin as in the proof of \eqref{eq:holomorphic-Eigenvalues-behave like-random-model1} and arrive at \eqref{eq:contribution-from-the-remainders}.
Using Lemma \ref{lemma: evenmomentlem-modular-2} and the condition $x P^{1000Y(1/\delta)^2 \sqrt{P}} < k/100$  , along with  $(2S)! \geq (2S/e)^{2S}$, we find this is all
\begin{eqnarray}
& \ll & x \sum_{i=1}^{Y} \frac{10^{2S}N|2\pi/\delta|^{2S}}{\delta S (2S)!} \Biggl(  (\frac{\sqrt{P}}{\log P})^{2S} \Biggr) \cdot \nonumber \\
&& \cdot \prod_{l=1}^{i-1} \Biggl(1 + O\Biggl(\frac{10^{2S}N|2\pi/\delta|^{2S}}{\delta S (2S)!} (\frac{\sqrt{P}}{\log P})^{2S}\Biggr) \Biggr) \nonumber \\
& \ll & x  \sum_{i=1}^{Y} \frac{N}{\delta S}  (\frac{10 e (\pi/\delta)^2  \sqrt{P}/\log P}{S})^{2S} \cdot \nonumber \\
&& \cdot \Biggl( 1 + O\Biggl(\frac{N}{\delta S} (\frac{10 e (\pi/\delta)^2  \sqrt{P}/\log P}{S})^{2S} \Biggr) \Biggr)^{i-1} . \nonumber
\end{eqnarray}
 One has the same overall bound for the contribution from the remainders $r_{j(i)}(\cdot)$ to the right hand side of \eqref{eq:holomorphic-Eigenvalues-behave like-random-model2}, since one has the same bound for $\E |\sum_{n \leq x} c(n) \X(n) |^{2} |Q(\X)|^{2k}$ as for the eigenvalue average in Lemma \ref{lemma: evenmomentlem-modular-2}.

Now if we set $S = 500\lfloor (1/\delta)^2 \sqrt{P} \rfloor$, then the condition $x^2 P^{4SY} < k^2/10^4$ is satisfied in view of our assumption that $x P^{1000Y(1/\delta)^2 \sqrt{P}} < k/100$. For $N\leq 10\log P$ and $Y\leq (\log P)^{100}$, the error term produced by all of the remainders is
$$ \ll x \sum_{i=1}^{Y} \frac{N}{\delta S}  0.6^S \left( 1 + O(\frac{N}{\delta S} 0.6^S ) \right)^{i-1} \ll x \cdot NY \cdot 0.6^S \ll \frac{x}{(N \log P)^{Y(1/\delta)^2}} , $$
as desired. Furthermore, with  $x P^{1000Y(1/\delta)^2 \sqrt{P}} < k/100$ and  $N \leq 10 \log P$, the error term in \eqref{eq:holomorphic-heckeeigenvalue-closeto-random-model-ploynomialcase}  is smaller than that stated in \eqref{eq:holomorphic-Eigenvalues-behave like-random-model2}.
The proof of \eqref{eq:holomorphic-Eigenvalues-behave like-random-model2} is now finished and  we have completed the proof of Proposition \ref{prop:holomorphic-Eigenvalues-behave like-random-model}. 
\end{proof}
We take $x=1$ and $c(1)=1$ in Proposition \ref{prop:holomorphic-Eigenvalues-behave like-random-model} to obtain the following special case.

\begin{corollary}\label{prop:holomorphic-Eigenvalues-behave like-random-model-specialcase}
Let functions $g_j$ ($-N\leq j\leq N+1$)  with associated parameters $N$ and $\delta$  be as  in Lemma \ref{lemma:partition-of-unity}. Let 
$ (j(i))_{i \leq Y} $
be any indices satisfying $-N \leq j(i)\leq N+1$ for $i\leq Y$, and let $(a_{i}(p^b))_{p \leq P}$ for $i\leq Y,~b=1,2$ be complex numbers such that $|a_{i}(p^b)|\leq 1$.  Let $P$ be large, and let $Y \in \N$.  Let $\X(n)$ denote a $\GL_2$ random multiplicative function. 

For 
$ P^{32000(Y/\delta)^2 \log(N\log P)} < k/100$, we have 
\begin{align}\label{eq:holomorphic-Eigenvalues-behave like-random-model-specialcase1}
&\ \ \ \ \E^{\rm Holo} \prod_{i=1}^{Y} g_{j(i)}(\Re(\sum_{p \leq P} \frac{a_{i}(p) \lambda_f(p)}{\sqrt{p}} + \frac{a_{i}(p^2) (\lambda_f(p^2)-1)}{p})) \\& =  \E \prod_{i=1}^{Y} g_{j(i)}(\Re(\sum_{p \leq P} \frac{a_{i}(p) \X(p)}{\sqrt{p}} + \frac{a_{i}(p^2) (\X(p^2)-1)}{p}))+ O\left(\frac{1}{(N \log P)^{2000Y(1/\delta)^2}} \right) . \nonumber
\end{align}
For $ P^{1000Y(1/\delta)^2 \sqrt{P}} < k/100$ , $N\leq 10\log P$ and $Y\leq (\log P)^{100}$ we have
\begin{align}
\label{eq:holomorphic-Eigenvalues-behave like-random-model-specialcase2}
&\ \ \  \E^{holo} \prod_{i=1}^{Y} g_{j(i)}(\Re(\sum_{p \leq P} \frac{a_{i}(p) \lambda_f(p)}{\sqrt{p}} + \frac{a_{i}(p^2) (\lambda_f(p^2)-1)}{p})) \\
& =  \E \prod_{i=1}^{Y} g_{j(i)}(\Re(\sum_{p \leq P} \frac{a_{i}(p) \X(p)}{\sqrt{p}} + \frac{a_{i}(p^2) (\X(p^2)-1)}{p}))  + O\left(\frac{1}{(N \log P)^{Y(1/\delta)^2}} \right) . \nonumber
\end{align}
\end{corollary}
Using Lemma \ref{lemma：evenmomentlem-maass}, we can obtain  similar results for  Hecke--Maass
cusp forms for \(\mathrm{SL}_2(\mathbb Z)\).
\begin{proposition}[Maass Hecke Eigenvalues behave like random model]\label{prop:Maass Eigenvalues behave like random model}
Let functions $g_j$ ($-N\leq j\leq N+1$)  with associated parameters $N$ and $\delta<1$  be as  in Lemma \ref{lemma:partition-of-unity}. Let $(c(n))_{n \leq x}$ be any sequence of complex numbers such that $|c(n)|\leq 1$. Let 
$ (j(i))_{i \leq Y} $
be any indices satisfying $-N \leq j(i)\leq N+1$ for $i\leq Y$, and let $(a_{i}(p^b))_{p \leq P}$ for $i\leq Y,~b=1,2$ be complex numbers such that $|a_{i}(p^b)|\leq 1$.  Suppose that $x \geq 1$. Let $P$ be large, and let $Y \in \N$.  Let $\X(n)$ denote a $\GL_2$ random multiplicative function. 

Let $\eta>0$ be a fixed small constant. For 
 $x P^{32000(Y/\delta)^2 \log(N\log P)} < T^{1-\eta}$, we have
\begin{align}
\label{eq:maass-heckeeigenvalue-closeto-random-model2}
&\ \ \  \E^{\rm Maass} \prod_{i=1}^{Y} g_{j(i)}(\Re(\sum_{p \leq P} \frac{a_{i}(p) \lambda_j(p)}{\sqrt{p}} + \frac{a_{i}(p^2) (\lambda_j(p^2)-1)}{p})) \Biggl|\sum_{n \leq x} c(n) \lambda_j(n) \Biggr|^2  \\
& =  \E \prod_{i=1}^{Y} g_{j(i)}(\Re(\sum_{p \leq P} \frac{a_{i}(p) \X(p)}{\sqrt{p}} + \frac{a_{i}(p^2) (\X(p^2)-1)}{p})) \Biggl|\sum_{n \leq x} c(n) \X(n) \Biggr|^2 + O\left(\frac{x\log (2x)}{(N \log P)^{Y(1/\delta)^2}} \right) . \nonumber
\end{align}
\end{proposition}

\begin{proof}[Proof of Proposition \ref{prop:Maass Eigenvalues behave like random model}]
As in the proof of Proposition \ref{prop:holomorphic-Eigenvalues-behave like-random-model}, we have $g_{j}(x) = \tilde{g}_{j}(x) + r_{j}(x)$, where $\tilde{g}_{j}(x)=
\sum_{k=0}^{2S-1}
\frac{g_{j}^{(k)}(0)}{k!}x^k
$ is a polynomial of degree $2S-1$ and $r_j(x)$ satisfies \eqref{eq:bound-for-remainder}. We expand all the polynomials and the square out and apply Lemma \ref{lemma: Eigenvalue like random model-maass} to find that
\begin{align}
\label{eq:maassheckeeigenvalue-closeto-random-model-ploynomialcase}
& \E^{\rm Maass} \prod_{i=1}^{Y} \tilde{g}_{j(i)}(\Re(\sum_{p \leq P} \frac{a_{i}(p) \lambda_j(p)}{\sqrt{p}} + \frac{a_{i}(p^2) (\lambda_j(p^2)-1)}{p})) \Biggl|\sum_{n \leq x} c(n) \lambda_j(n) \Biggr|^2  \\
= &  \E \prod_{i=1}^{Y} \tilde{g}_{j(i)}(\Re( \sum_{p \leq P} \frac{a_{i}(p) \X(p)}{\sqrt{p}} + \frac{a_{i}(p^2) (\X(p^2)-1)}{p})) \Biggl|\sum_{n \leq x} c(n) \X(n) \Biggr|^2\nonumber \\ &  +  O\left(
x^{2+\epsilon}\left(\frac{7N}{\delta}e^{64\pi/\delta}P^{(1+4\varepsilon)S}\right)^Y
T^{-(1-\varepsilon)}\right)\nonumber \\& 
+  O\left(
x^{5/2+\epsilon}\left(\frac{7N}{\delta}e^{64\pi/\delta}P^{(3/2+4\varepsilon)S}\right)^Y
T^{-2}\right). \nonumber
\end{align}

Next, dividing up according to the smallest index $i$ at which we get a remainder, we see the contribution from all of the remainders $r_{j(i)}(\cdot)$ to the left hand side in Proposition \ref{prop:Maass Eigenvalues behave like random model} is
\begin{align}
 \ll &\  \E^{\rm Maass} \sum_{i=1}^{Y} \frac{N |2\pi/\delta|^{2S}}{\delta S (2S)!} |\sum_{p \leq P} \frac{a_{i}(p) \lambda_j(p)}{\sqrt{p}} + \frac{a_{i}(p^2) (\lambda_j(p^2)-1)}{p}|^{2S} \cdot \nonumber \\
& \cdot \prod_{l=1}^{i-1} \Biggl(1 + O(\frac{N |2\pi/\delta|^{2S}}{\delta S (2S)!} |\sum_{p \leq P} \frac{a_{l}(p) \lambda_j(p)}{\sqrt{p}} + \frac{a_{l}(p^2) (\lambda_j(p^2)-1)}{p}|^{2S})\Biggr) \Biggl|\sum_{n \leq x} c(n) \lambda_j(n) \Biggr|^2 . \nonumber
\end{align}
Here we used the fact that $|\tilde{g}_{j(l)}(x)| \leq |g_{j(l)}(x)| + |r_{j(l)}(x)| \leq 1 + O(\frac{N |2\pi x/\delta|^{2S}}{\delta S (2S)!})$. Using Lemma \ref{lemma：evenmomentlem-maass} and the condition $xP^{2SY}\leq T^{1-\eta}$ , along with  $(2S)! \geq (2S/e)^{2S}$, we find this is all
\begin{align}
& \ll  x\log x \cdot \sum_{i=1}^{Y} \frac{N|2\pi/\delta|^{2S}}{\delta S (2S)!} \Biggl( \sqrt{iS} (\frac{iS}{e})^S (320\log\log P)^{S} \Biggr) \cdot \nonumber \\
&\ \  \cdot \prod_{l=1}^{i-1} \Biggl(1 + O\Biggl(\frac{N|2\pi/\delta|^{2S}}{\delta S (2S)!} (\frac{iS}{e})^S (320\log\log P)^{S} \Biggr) \Biggr) \nonumber \\
& \ll  x\log x \cdot \sum_{i=1}^{Y} \frac{N}{\delta} \sqrt{\frac{i}{S}} (\frac{e (\pi/\delta)^2 i (320\log\log P)}{S})^S \cdot \nonumber \\
&\ \  \cdot \Biggl( 1 + O\Biggl(\frac{N}{\delta S} (\frac{e (\pi/\delta)^2 i (320\log\log P)}{S})^S \Biggr) \Biggr)^{i-1} . \nonumber
\end{align}
 One has the same overall bound for the contribution from the remainders $r_{j(i)}(\cdot)$ to the right hand side in Proposition \ref{prop:Maass Eigenvalues behave like random model}, since one has the same bound for $\E |\sum_{n \leq x} c(n) \X(n) |^{2} |Q(\X)|^{2k}$ as for the eigenvalue average in Lemma \ref{lemma：evenmomentlem-maass}.

Now  we set $S = 16000Y \lfloor (1/\delta)^2 \log(N\log P) \rfloor$. Then the condition $xP^{2SY}\leq T^{1-\eta}$  is satisfied in view of our assumption that $x P^{32000(Y/\delta)^2 \log(N\log P)} < T^{1-\eta}$. The error term produced by all of the remainders is
$$ \ll x\log (2x)\cdot \sum_{i=1}^{Y} \frac{N}{\delta} \sqrt{\frac{i}{S}} 0.6^S \left( 1 + O(\frac{N}{\delta S} 0.6^S ) \right)^{i-1} \ll x\log (2x) \cdot NY \cdot 0.6^S \ll \frac{x\log (2x)}{(N \log P)^{Y(1/\delta)^2}} , $$
as desired. Furthermore, with  $x P^{32000(Y/\delta)^2 \log(N\log P)} < T^{1-\eta}$, the error term in \eqref{eq:maassheckeeigenvalue-closeto-random-model-ploynomialcase}  is smaller than that stated in \eqref{eq:maass-heckeeigenvalue-closeto-random-model2}.
The proof of Proposition \ref{prop:Maass Eigenvalues behave like random model} is now finished.
\end{proof}

\section{Proof of Theorem \ref{mainthm:low-moment-random-GL2} and  Corollary \ref{Cor:limiting-behavior-of-sums of-Hecke-eigenvalues}}
\begin{proof}[Proof of Theorem \ref{mainthm:low-moment-random-GL2}]The proof follows the method of Harper \cite[Theorem~2]{H20}. His proof may be divided into two steps. The first is to reduce the problem to estimating certain expectations involving random Euler product; see \cite[Propositions~2 and~4]{H20}. For the upper bound, this reduction relies essentially on the orthogonality relations satisfied by Rademacher random multiplicative function. Since the $\GL_2$ random multiplicative function also satisfies orthogonality relations, the argument of \cite[Proposition 2]{H20} carries over to the $\GL_2$ setting. For the lower bound, this reduction uses Khintchine’s inequality, which also holds for independent $\GL_2$ random variables.

The second step is to estimate the resulting expectations of random Euler products. This step requires two kinds of probabilistic estimates. One concerns the mean square and related quantities of certain random Euler products, while the other concerns the probability of the barrier event that the logarithms
of certain Euler subproducts satisfy prescribed upper and lower bounds. The latter estimates, which are  key ingredients in this step, are based on the former. Lemmas \ref{lemma:mean-square-randomEulerproducts-upperbd} and \ref{lemma:mean-square-randomEulerproducts-lowerbd}, when compared with \cite[Lemma 2 and Section~5.3]{H20}, show that the mean square and related
quantities of certain random Euler products associated with the $\GL_2$ random multiplicative function agree with their counterparts in the Rademacher case. More precisely, Lemma \ref{lemma:mean-square-randomEulerproducts-upperbd}, which corresponds to \cite[Lemma 2]{H20}, is used in the proof of the upper bound, while Lemma \ref{lemma:mean-square-randomEulerproducts-lowerbd}, which corresponds to the computation in \cite[Section 5.3]{H20}, is used in the proof of the lower bound. Therefore, we can establish the same probabilistic estimates for the $\GL_2$ random multiplicative function as in the Rademacher case. Once these estimates have been established, the remainder of Harper's argument for Rademacher random multiplicative functions carries over to the $\GL_2$ case.
\end{proof}

\begin{proof}[Proof of Corollary \ref{Cor:limiting-behavior-of-sums of-Hecke-eigenvalues}]
The argument here is taken from \cite{HL16}.
 In view of Theorem \ref{mainthm:low-moment-random-GL2},  to prove \eqref{eq:limiting-behavior-of-sums of-holomorphic-Hecke-eigenvalues}, it suffices to prove  that for any  $q>0$,
 \begin{align*}
       \lim_{k \to \infty } 
\sumh_{f \in \mathcal{H}_k}
\left| \sum_{n \leq x} \lambda_f(n) \right|^{2q}=\mathbb{E}\left|
        \sum_{n\leq x}\mathbb{X}(n)
        \right|^{2q}.
 \end{align*}
 For $q\in \mathbb{N}$, this follows from Lemma \ref{lemma: Eigenvalue like random model-holomorphic}. Then using Weierstrass approximation to the function
$f:y\mapsto y^{q/2}$, one can prove this for any $q>0$. Similarly,  using Lemma \ref{lemma: Eigenvalue like random model-maass} we can prove that for any  $q>0$
 \begin{align}\label{eq:limiting-behavior-of-sums of-Maass-Hecke-eigenvalues-anotherform}
       \lim_{T \to \infty } 
\frac{12}{T^2}\sum_{j}
        \frac{\zeta(2)}{L(1,\mathrm{sym}^2 u_j)}e^{-t_j/T}\left|
        \sum_{n\leq x}\lambda_j(n)
        \right|^{2q}=\mathbb{E}\left|
        \sum_{n\leq x}\mathbb{X}(n)
        \right|^{2q}.
 \end{align}
Then \eqref{eq:limiting-behavior-of-sums of-Maass-Hecke-eigenvalues} follows from \eqref{eq:limiting-behavior-of-sums of-Maass-Hecke-eigenvalues-anotherform} by using  a standard Tauberian theorem (see the proof of \cite[Theorem 2]{BBR14}).
\end{proof}

\section{Proof of Theorem \ref{mainthm:low-moment-modular-sharp}, \ref{mainthm:low-moment-modular}, \ref{mainthm:low-moment-maass} and Corollary \ref{cor:quantitative-corollary}}
We first present the proof of Theorem \ref{mainthm:low-moment-modular-sharp}.
\subsection{Random Euler products}\label{subsecrandeuler}
Let \(\mathbb{X}(n)\) be a \(\GL_2\) random multiplicative function.
For any large quantity $P$ and any complex $s$ with $\Re(s) > 0$, let $F_{P}^{\text{rand}}(s) := \prod_{p \leq P} (1 - \frac{\X(p)}{p^s}+\frac{1}{p^{2s}})^{-1}$.
We have the following the following result analogue to \cite[Multiplicative Chaos Result 1]{H23}
\begin{lemma}\label{lemma: Multiplicative Chaos Result1}
Uniformly for all large $P$ and $2/3 \leq q \leq 1$, we have
$$ \frac{1}{(|v|+1)^{q/4}}\E( \int_{v-1/2}^{v+1/2} |F_{P}^{\text{rand}}(1/2 + it)|^2 dt )^q \ll \Biggl(\frac{\log P}{1 + (1-q)\sqrt{\log\log P}}\Biggr)^{q} . $$
\end{lemma}

The analogous  conclusion for Rademacher random multiplicative functions was proved in Section~4 of Harper~\cite{H20}; see the proof of the upper bound in Theorem~2 there. Since the proof of Theorem~\ref{mainthm:low-moment-random-GL2} is analogous to the proof of Theorem~2 in Harper~\cite{H20}, Lemma~\ref{lemma: Multiplicative Chaos Result1} follows by the same argument.

We now deduce a discrete version of  Lemma \ref{lemma: Multiplicative Chaos Result1} using arguments as in the proof \cite[Multiplicative Chaos Result 2]{H23}

\begin{lemma}\label{disccontlem}
Let $v\in \mathbb{Z}$. For any large $P$, we have
$$ \E \sum_{v-1/2\leq \frac{k}{\log^{1.01}P} \leq v+1/2} \int_{-\frac{1}{2\log^{1.01}P}}^{\frac{1}{2\log^{1.01}P}} |F_{P}^{\text{rand}}(1/2 + i\frac{k}{\log^{1.01}P} + it) - F_{P}^{\text{rand}}(1/2 + i\frac{k}{\log^{1.01}P})|^2 dt \ll \log^{0.99}P . $$
\end{lemma}

\begin{proof}[Proof of Lemma \ref{disccontlem}]
Since we have $F_{P}^{\text{rand}}(s) = \sum_{\substack{n=1, \\  n \; \text{is} \; P \; \text{smooth}}}^{\infty} \frac{\X(n)}{n^s}$, with $\X(n)$ a $GL_2$ random multiplicative function, a mean square calculation shows that the left hand side in Lemma \ref{disccontlem} is
\begin{eqnarray}
& = & \sum_{v-1/2\leq \frac{k}{\log^{1.01}P} \leq v+1/2} \int_{-\frac{1}{2\log^{1.01}P}}^{\frac{1}{2\log^{1.01}P}} \E |\sum_{\substack{n=1, \\  n \; \text{is} \; P \; \text{smooth}}}^{\infty} \frac{\X(n)}{n^{1/2 + i\frac{k}{\log^{1.01}P}}} (n^{-it} - 1)|^2 dt \nonumber \\
& = & \sum_{v-1/2\leq \frac{k}{\log^{1.01}P} \leq v+1/2} \int_{-\frac{1}{2\log^{1.01}P}}^{\frac{1}{2\log^{1.01}P}} \sum_{\substack{n=1, \\  n \; \text{is} \; P \; \text{smooth}}}^{\infty} \frac{|n^{-it}-1|^2}{n} dt . \nonumber
\end{eqnarray}
Here we have $|n^{-it} - 1| \ll \min\{|t|\log n, 1\} \leq \min\{\frac{\log n}{\log^{1.01}P}, 1\}$. Thus the contribution to the series from those $n \leq P^{\log\log P}$ is $\ll \frac{(\log\log P)^2}{\log^{0.02}P} \sum_{\substack{n=1, \\  n \; \text{is} \; P \; \text{smooth}}}^{\infty} \frac{1}{n}$, and since $\sum_{\substack{n=1, \\  n \; \text{is} \; P \; \text{smooth}}}^{\infty} \frac{1}{n} = \prod_{p \leq P} (1 - \frac{1}{p})^{-1} \ll \log P$ this is all $\ll (\log\log P)^2 \log^{0.98}P$. The contribution to the series from those $n > P^{\log\log P}$ (where we use the trivial bound $|n^{-it} - 1| \ll 1$) is also acceptably small, namely
$$ \ll e^{-\log\log P} \sum_{\substack{n=1, \\  n \; \text{is} \; P \; \text{smooth}}}^{\infty} \frac{1}{n^{1-1/\log P}} = e^{-\log\log P} \prod_{p \leq P} (1 - \frac{1}{p^{1-1/\log P}})^{-1} \ll e^{-\log\log P} \log P = 1 . $$
\end{proof}

\begin{lemma}\label{lemma: Multiplicative Chaos Result2}
Let $v\in \mathbb{Z}$.
Uniformly for all large $P$ and $2/3 \leq q \leq 1$, we have
$$ \frac{1}{(|v|+1)^{q/4}} \E( \frac{1}{\log^{1.01}P} \sum_{v-1/2\leq \frac{k}{\log^{1.01}P} \leq v+1/2} |F_{P}^{\text{rand}}(1/2 + i\frac{k}{\log^{1.01}P})|^2 )^q \ll \Biggl(\frac{\log P}{1 + (1-q)\sqrt{\log\log P}}\Biggr)^{q} . $$
\end{lemma}

\begin{proof}
To deduce this from Lemma \ref{lemma: Multiplicative Chaos Result1}, it will suffice to prove a suitable upper bound for
$$ \E( \sum_{v-1/2\leq \frac{k}{\log^{1.01}P} \leq v+1/2} \int_{-1/(2\log^{1.01}P)}^{1/(2\log^{1.01}P)} |F_{P}^{\text{rand}}(1/2 + i\frac{k}{\log^{1.01}P} + it) - F_{P}^{\text{rand}}(1/2 + i\frac{k}{\log^{1.01}P})|^2 dt )^q . $$
But using H\"{o}lder's inequality to compare the $q$-th moment with the first moment, we see this quantity is at most
$$ \Biggl( \E \sum_{v-1/2\leq \frac{k}{\log^{1.01}P} \leq v+1/2} \int_{-\frac{1}{2\log^{1.01}P}}^{\frac{1}{2\log^{1.01}P}} |F_{P}^{\text{rand}}(1/2 + i\frac{k}{\log^{1.01}P} + it) - F_{P}^{\text{rand}}(1/2 + i\frac{k}{\log^{1.01}P})|^2 dt \Biggr)^q , $$
which we can bound acceptably using Lemma \ref{disccontlem}. 
\end{proof}
\subsection{Preliminary reduction}\label{thm1small}
It suffices to consider the range \(2/3\leq q\leq 1\). Indeed, if
\(0\leq    q<2/3\), then H\"older's inequality gives
\[
\E^{\rm holo}
\left|
\sum_{n\leq x}\lambda_f(n)
\right|^{2q}
\ll
\left(
\E^{\rm holo}
\left|
\sum_{n\leq x}\lambda_f(n)
\right|^{4/3}
\right)^{3q/2},
\]
so the desired estimate follows from the case \(q=2/3\).
We may assume that
\(
L_k:=\min\{x,k/x\}
\)
is  large, since otherwise the conclusion of
Theorem~\ref{mainthm:low-moment-modular-sharp} is trivial. Let \(P\) be the
largest  number not exceeding
\(
\exp\bigl((\log L_k)^{1/6}\bigr)
\)
such that \((\log P)^{0.01}\) is an integer. It follows that
\(
\log P\asymp(\log L_k)^{1/6}
\) and \(
\log\log P\asymp\log\log L_k.
\)
Consequently, to prove Theorem~\ref{mainthm:low-moment-modular-sharp}, it is
enough to show that
\begin{equation}\label{stpmaindisplay}
\E^{\rm holo}
\left|
\sum_{n\leq x}\lambda_f(n)
\right|^{2q}
\ll
\left(
\frac{x}
{1+(1-q)\sqrt{\log\log P}}
\right)^q.
\end{equation}
Set $M := 2\log^{1.02}P$. For each integer $m$ satisfying $|m| \leq M$ and each  \(f\in \mathcal{H}_k\), we set \[S_{m}(f) := \Re \sum_{p \leq P} (\frac{\lambda_f(p)}{p^{1/2 + im/\log^{1.01}P}} + \frac{\lambda_f(p^2)-1}{2p^{1 + 2im/\log^{1.01}P}}).\]
Let $\X$ be a $\GL_2$ random multiplicative function.  Set  
\[S_{m}(\X) := \Re \sum_{p \leq P} (\frac{\X(p)}{p^{1/2 + im/\log^{1.01}P}} + \frac{\X(p^2)-1}{2p^{1 + 2im/\log^{1.01}P}}).\]
\subsection{The conditioning argument}\label{subsecmaincond}
Let $P(n)$ denote the largest prime factor of $n$, and let $\Psi(x,y) := \#\{n \leq x : P(n) \leq y \}$. It follows from H\"{o}lder's inequality and Lemma \ref{lemma:orthogonality-holomorphic} that
\begin{eqnarray}
&&\E^{\rm holo} |\sum_{n \leq x, P(n) \leq x^{1/\log\log x}} \lambda_f(n)|^{2q} \\ & \leq &\Biggl(\E^{\rm holo} |\sum_{n \leq x, P(n) \leq x^{1/\log\log x}} \lambda_f(n)|^{2} \Biggr)^q = (\Psi(x,x^{1/\log\log x})+O(x^2e^{-k}))^q.\nonumber\end{eqnarray} 
Since $x\leq k$, the term $O(x^2e^{-k})$ is very small. By smooth number estimates (see Montgomery and Vaughan~\cite[Theorem 7.6]{MV07}), the last line above  is $\ll (x (\log x)^{-c\log\log\log x})^q$, which gives a negligible contribution in Theorem \ref{mainthm:low-moment-modular-sharp}.

Thus, it suffices to estimate $\E^{\rm holo} |\sum_{n \leq x, P(n) > x^{1/\log\log x}} \lambda_f(n)|^{2q}$.  The partition of unity  $g_j$ given by lemma \ref{lemma:partition-of-unity} allow us to rewrite $\E^{\rm holo} |\sum_{n \leq x, P(n) > x^{1/\log\log x}} \lambda_f(n)|^{2q}$ as
\begin{eqnarray}
&& \E^{\rm holo} \prod_{i=-M}^{M} (\sum_{j=-N}^{N+1} g_{j}(S_{i}(f))) \Biggl|\sum_{\substack{n \leq x, \\ P(n) > x^{1/\log\log x}}} \lambda_f(n) \Biggr|^{2q} \nonumber \\
& = & \sum_{-N \leq j(-M), ..., j(0), ..., j(M) \leq N+1} \E^{\rm holo} \prod_{i=-M}^{M} g_{j(i)}(S_{i}(f)) \Biggl|\sum_{\substack{n \leq x, \\ P(n) > x^{1/\log\log x}}} \lambda_f(n) \Biggr|^{2q} \nonumber \\
& = & \sum_{-N \leq j(-M) , ... , j(0), ..., j(M) \leq N+1} \sigma(\textbf{j}) \E^{\textbf{j}} \Biggl|\sum_{\substack{n \leq x, \\ P(n) > x^{1/\log\log x}}} \lambda_f(n) \Biggr|^{2q} , \nonumber
\end{eqnarray}
where for any $(2M+1)$-vector $\textbf{j}=(j(-M) , ... , j(0), ..., j(M))$ and  any function $W(f)$, we set $$\sigma(\textbf{j}) := \E^{\rm holo} \prod_{i=-M}^{M} g_{j(i)}(S_{i}(f)),~ \E^{\textbf{j}} W := \sigma(\textbf{j})^{-1} \E^{\rm holo} W \prod_{i=-M}^{M} g_{j(i)}(S_{i}(f)).$$  Note that for the constant function 1, we have $\E^{\textbf{j}} 1 = 1$ for all choices of the vector $\textbf{j}$. Then we apply H\"older's inequality to $\E^{\textbf{j}}$ and conclude that
\begin{align}
\label{eq:holder-to-Ej}
\E^{\rm holo}\Biggl|\sum_{\substack{n \leq x, \\ P(n) > x^{1/\log\log x}}} \lambda_f(n) \Biggr|^{2q} \leq \sum_{-N \leq j(-M) , ..., j(0), ..., j(M) \leq N+1} \sigma(\textbf{j})\Biggl( \E^{\textbf{j}} \Biggl|\sum_{\substack{n \leq x, \\ P(n) > x^{1/\log\log x}}} \lambda_f(n) \Biggr|^{2} \Biggr)^{q}. \end{align} 

\subsection{Passing to the random case}\label{subsecmaingorandom} Assume for the moment that our eventual choices of $N$ and  $\delta$ satisfy $x P^{32000((2M+1)/\delta)^2 \log(N\log P)} < k/100$.
Then applying \eqref{eq:holomorphic-Eigenvalues-behave like-random-model1} in Proposition
\ref{prop:holomorphic-Eigenvalues-behave like-random-model} with $Y=2M+1$, we obtain
\begin{align}\label{eq:prop-hecke-random}
&\ \ \ \ \E^{\textbf{j}} \Biggl|\sum_{\substack{n \leq x, \\ P(n) > x^{1/\log\log x}}} \lambda_f(n) \Biggr|^{2} \\&= \frac{1}{\sigma(\textbf{j})} \Biggl( \E \prod_{i=-M}^{M} g_{j(i)}(S_{i}(\X)) \Biggl|\sum_{\substack{n \leq x, \\ P(n) > x^{1/\log\log x}}} \X(n) \Biggr|^{2} + O\left(\frac{x\log x}{(N \log P)^{2000(2M+1)(1/\delta)^2}} \right) \Biggr)\nonumber.
\end{align}  
Since the functions $g_j$ form a partition of unity, we have $$\sum_{\textbf{j}} \sigma(\textbf{j}) = \E^{\rm Holo} \prod_{i=-M}^{M} (\sum_{j=-N}^{N+1} g_{j}(S_{i}(f))) = \E^{\rm Holo} 1 = 1+O(e^{-k}).$$ By 
H\"{o}lder's inequality applied to $\sum_{\textbf{j}}$, the total contribution of the error
terms in \eqref{eq:prop-hecke-random} to $\E |\sum_{n \leq x, P(n) > x^{1/\log\log x}} \lambda_f(n)|^{2q}$ is
\begin{equation}\label{eq:contribution-of-error-terms}
\begin{split}
  & \ll \Biggl( \sum_{-N \leq j(-M) , ..., j(0), ..., j(M) \leq N+1} \sigma(\textbf{j}) \cdot \frac{1}{\sigma(\textbf{j})} \frac{x\log x}{(N \log P)^{2000(2M+1)(1/\delta)^2}} \Biggr)^{q} \\& = \Biggl( x \log x(\frac{(2N+2)}{(N \log P)^{2000(1/\delta)^2}})^{2M+1} \Biggr)^{q} .  
\end{split}
    \end{equation}
This is negligible in comparison with the right-hand side of
\eqref{stpmaindisplay}, provided that
\[
\log x
\leq
(N\log P)^{1000(2M+1)(1/\delta)^2}.
\]

Next, we define \[\sigma^{\text{rand}}(\textbf{j}) := \E \prod_{i=-M}^{M} g_{j(i)}(S_{i}(\X))\]for all $(2M+1)$-vectors $\textbf{j}$, where $\X$ is a $\GL_2$ random multiplicative function.  We apply \eqref{eq:holomorphic-Eigenvalues-behave like-random-model-specialcase1} in  Propositions \ref{prop:holomorphic-Eigenvalues-behave like-random-model-specialcase} to obtain $\sigma(\textbf{j})^{1-q} \ll \sigma^{\text{rand}}(\textbf{j})^{1-q} + (\frac{1}{(N \log P)^{(2M+1)(1/\delta)^2}})^{1-q}$. Thus, we have 
\begin{eqnarray}
&& \sum_{-N \leq j(-M) , ..., j(0), ..., j(M) \leq N+1} \sigma(\textbf{j})\Biggl( \frac{1}{\sigma(\textbf{j})} \E \prod_{i=-M}^{M} g_{j(i)}(S_{i}(\X)) \Biggl|\sum_{\substack{n \leq x, \\ P(n) > x^{1/\log\log x}}} \X(n) \Biggr|^{2} \Biggr)^q \nonumber \\
& \ll & \sum_{\textbf{j}} \sigma^{\text{rand}}(\textbf{j})\Biggl( \frac{1}{\sigma^{\text{rand}}(\textbf{j})} \E \prod_{i=-M}^{M} g_{j(i)}(S_{i}(\X)) \Biggl|\sum_{\substack{n \leq x, \\ P(n) > x^{1/\log\log x}}} \X(n) \Biggr|^{2} \Biggr)^q \nonumber \\
&& + (\frac{1}{(N \log P)^{(2M+1)(1/\delta)^2}})^{1-q} \sum_{\textbf{j}} \Biggl( \E \prod_{i=-M}^{M} g_{j(i)}(S_{i}(\X)) \Biggl|\sum_{\substack{n \leq x, \\ P(n) > x^{1/\log\log x}}} \X(n) \Biggr|^{2} \Biggr)^q . \nonumber
\end{eqnarray}
Applying H\"older's inequality again to the sum over $\textbf{j}$, and recalling that the $g_j$ form a partition of unity, we have
\begin{eqnarray}
&&(\frac{1}{(N \log P)^{(2M+1)(1/\delta)^2}})^{1-q} \sum_{\textbf{j}} \Biggl( \E \prod_{i=-M}^{M} g_{j(i)}(S_{i}(\X)) \Biggl|\sum_{\substack{n \leq x, \\ P(n) > x^{1/\log\log x}}} \X(n) \Biggr|^{2} \Biggr)^q \nonumber\\
& \ll & (\frac{1}{(N \log P)^{(2M+1)(1/\delta)^2}})^{1-q} \cdot ((2N+2)^{2M+1})^{1-q} \cdot \Biggl( \sum_{\textbf{j}} \E \prod_{i=-M}^{M} g_{j(i)}(S_{i}(\X)) \Biggl|\sum_{\substack{n \leq x, \\ P(n) > x^{1/\log\log x}}} \X(n) \Biggr|^{2} \Biggr)^q \nonumber \\
& = & \Biggl( (\frac{(2N+2)}{(N \log P)^{(1/\delta)^2}})^{2M+1} \Biggr)^{1-q} \Biggl( \E \Biggl|\sum_{\substack{n \leq x, \\ P(n) > x^{1/\log\log x}}} \X(n) \Biggr|^{2} \Biggr)^q \leq (\log P)^{-(1-q)} x^q. \nonumber
\end{eqnarray}
which gives a negligible contribution to \eqref{stpmaindisplay}.

We now define $\E^{\textbf{j}, \text{rand}} W := \sigma^{\text{rand}}(\textbf{j})^{-1} \E W \prod_{i=-M}^{M} g_{j(i)}(S_{i}(\X))$ for all random variables $W$. To prove Theorem \ref{mainthm:low-moment-modular-sharp} it remains to show that
\begin{equation}\label{eq:what-remains-to-prove-after-pass-to-RMF}
\sum_{-N \leq j(-M) , ..., j(0), ..., j(M) \leq N+1} \sigma^{\text{rand}}(\textbf{j}) \Biggl( \E^{\textbf{j}, \text{rand}} \Biggl|\sum_{\substack{n \leq x, \\ P(n) > x^{1/\log\log x}}} \X(n) \Biggr|^{2} \Biggr)^{q} \ll \Biggl(\frac{x}{1 + (1-q)\sqrt{\log\log P}} \Biggr)^q .
\end{equation}

\subsection{Passing to Euler products}\label{subsecpasseuler} With $P< x^{1/\log\log x}$, we can use the same argument as in section 3.4 of Harper~\cite{H23} to show that, the left hand side of \eqref{eq:what-remains-to-prove-after-pass-to-RMF} is 

\begin{equation}\label{produpperintrand}
\ll (\frac{x}{\log P})^q \sum_{\textbf{j}} \sigma^{\text{rand}}(\textbf{j}) (\E^{\textbf{j}, \text{rand}} \int_{-\infty}^{\infty} \frac{|F_{P}^{\text{rand}}(1/2 + it)|^2}{|1/2+it|^2} dt)^q ,
\end{equation}
where \begin{align*}
F_{P}^{\text{rand}}(s) := \prod_{p \leq P} (1 - \frac{\X(p)}{p^s}+\frac{1}{p^{2s}})^{-1}=\sum_{\substack{n=1, \\  n \; \text{is} \; P \; \text{smooth}}}^{\infty} \frac{\X(n)}{n^s}
\end{align*} is the Euler product associated to $\X(n)$ on $P$-smooth numbers. 

Since we have restricted to the range $2/3 \leq q \leq 1$, we may decompose the integral in \eqref{produpperintrand} into sub-intervals of length 1 and so that the expression in \eqref{produpperintrand} is
\begin{eqnarray}
& \leq & (\frac{x}{\log P})^q \sum_{\textbf{j}} \sigma^{\text{rand}}(\textbf{j}) \sum_{v=-\infty}^{\infty} (\E^{\textbf{j}, \text{rand}} \int_{v-1/2}^{v+1/2} \frac{|F_{P}^{\text{rand}}(1/2 + it)|^2}{|1/2+it|^2} dt)^q \nonumber \\
& \ll & (\frac{x}{\log P})^q \sum_{v=-\infty}^{\infty} \frac{1}{(|v|+1)^{2q}} \sum_{\textbf{j}} \sigma^{\text{rand}}(\textbf{j}) (\E^{\textbf{j}, \text{rand}} \int_{v-1/2}^{v+1/2} |F_{P}^{\text{rand}}(1/2 + it)|^2 dt)^q . \nonumber
\end{eqnarray} 
We first deal with the terms with
\(|v|>\log^{0.01}P\).
We apply H\"{o}lder's inequality  and the orthogonality of random multiplicative functions to see their contribution is
\begin{eqnarray}
& \leq & (\frac{x}{\log P})^q \sum_{|v| > \log^{0.01}P} \frac{1}{(|v|+1)^{4/3}} (\sum_{\textbf{j}} \sigma^{\text{rand}}(\textbf{j}) \E^{\textbf{j}, \text{rand}} \int_{v-1/2}^{v+1/2} |F_{P}^{\text{rand}}(1/2 + it)|^2 dt)^q \nonumber \\
& = & (\frac{x}{\log P})^q \sum_{|v| > \log^{0.01}P} \frac{1}{(|v|+1)^{4/3}} (\E \int_{v-1/2}^{v+1/2} |F_{P}^{\text{rand}}(1/2 + it)|^2 dt)^q , \nonumber\\
& = & (\frac{x}{\log P})^q \sum_{|v| > \log^{0.01}P} \frac{1}{(|v|+1)^{4/3}} ( \sum_{\substack{n = 1, \\ n \; \text{is} \; P \; \text{smooth}}}^{\infty} \frac{1}{n} )^q \ll (\frac{x}{\log P})^q \frac{1}{\log^{1/300}P} \log^{q}P , \end{eqnarray}
which is negligible for \eqref{eq:what-remains-to-prove-after-pass-to-RMF}.

In view of \eqref{eq:what-remains-to-prove-after-pass-to-RMF} and \eqref{produpperintrand}, it is sufficient to show that, uniformly for all $|v| \leq \log^{0.01}P$, we have
\begin{equation}\label{eq:goal-euler-product-upperbd}
\frac{1}{(|v|+1)^{q/4}} \sum_{\textbf{j}} \sigma^{\text{rand}}(\textbf{j}) (\E^{\textbf{j}, \text{rand}} \int_{v-1/2}^{v+1/2} |F_{P}^{\text{rand}}(1/2 + it)|^2 dt)^q \ll \Biggl(\frac{\log P}{1 + (1-q)\sqrt{\log\log P}}\Biggr)^{q} .
\end{equation}

\subsection{Conclusion}\label{subsecrefinecond} We focus on the case $v=0$ in \eqref{eq:goal-euler-product-upperbd} since the treatments of the remaining cases are similar.
In this case, we have
\begin{eqnarray}
&& \sum_{\textbf{j}} \sigma^{\text{rand}}(\textbf{j}) (\E^{\textbf{j}, \text{rand}} \int_{-1/2}^{1/2} |F_{P}^{\text{rand}}(1/2 + it)|^2 dt)^q \nonumber \\
& \ll & \sum_{\textbf{j}} \sigma^{\text{rand}}(\textbf{j}) ( \E^{\textbf{j}, \text{rand}} \sum_{|k| \leq \frac{\log^{1.01}P}{2}} \int_{-\frac{1}{2\log^{1.01}P}}^{\frac{1}{2\log^{1.01}P}} |F_{P}^{\text{rand}}(\frac{1}{2} + i\frac{k}{\log^{1.01}P}+it) - F_{P}^{\text{rand}}(\frac{1}{2} + i\frac{k}{\log^{1.01}P})|^2 dt )^q \nonumber \\
&& + \sum_{\textbf{j}} \sigma^{\text{rand}}(\textbf{j}) (\E^{\textbf{j}, \text{rand}} \frac{1}{\log^{1.01}P} \sum_{|k| \leq (\log^{1.01}P)/2} |F_{P}^{\text{rand}}(1/2 + i\frac{k}{\log^{1.01}P})|^2 )^q . \nonumber
\end{eqnarray}
Applying H\"older's inequality to the sum over $\mathbf{j}$ and Lemma \ref{disccontlem}, we see  the
first term is
$$ \leq \Biggl( \E \sum_{|k| \leq \frac{\log^{1.01}P}{2}} \int_{-\frac{1}{2\log^{1.01}P}}^{\frac{1}{2\log^{1.01}P}} |F_{P}^{\text{rand}}(1/2 + i\frac{k}{\log^{1.01}P} + it) - F_{P}^{\text{rand}}(1/2 + i\frac{k}{\log^{1.01}P})|^2 dt \Biggr)^q \ll \log^{0.99q}P, $$
which gives a negligible contribution to \eqref{eq:goal-euler-product-upperbd}. Thus, it remains to handle the the second term.

We now define
$$ \mathcal{T} := \{k \in \Z : |F_{P}^{\text{rand}}(1/2 + i\frac{k}{\log^{1.01}P})| \geq \log^{1.1}P \;\;\; \text{or} \;\;\; |F_{P}^{\text{rand}}(1/2 + i\frac{k}{\log^{1.01}P})| \leq \frac{1}{\log^{1.1}P} \} , $$
say.  We split the sum over $|m| \leq (\log^{1.01}P)/2$ according to whether $m \in \mathcal{T}$ or not, and then apply H\"{o}lder's inequality to obtain
\begin{eqnarray}
&& \sum_{\textbf{j}} \sigma^{\text{rand}}(\textbf{j}) (\E^{\textbf{j}, \text{rand}} \frac{1}{\log^{1.01}P} \sum_{|k| \leq (\log^{1.01}P)/2} |F_{P}^{\text{rand}}(1/2 + i\frac{k}{\log^{1.01}P})|^2 )^q \nonumber \\
& \leq & \sum_{\textbf{j}} \sigma^{\text{rand}}(\textbf{j}) (\E^{\textbf{j}, \text{rand}} \frac{1}{\log^{1.01}P} \sum_{\substack{|k| \leq (\log^{1.01}P)/2, \\ k \notin \mathcal{T}}} |F_{P}^{\text{rand}}(1/2 + i\frac{k}{\log^{1.01}P})|^2 )^q + \nonumber \\
&& + \Biggl( \frac{1}{\log^{1.01}P} \sum_{|k| \leq (\log^{1.01}P)/2} \E \textbf{1}_{k \in \mathcal{T}} |F_{P}^{\text{rand}}(1/2 + i\frac{k}{\log^{1.01}P})|^2 \Biggr)^q . \nonumber
\end{eqnarray}
The contribution form those $k$ with $|t_k:=k/\log^{1.01}P| \leq 1/\log\log P$ to the second term on the right hand side is at most
 \[
\begin{aligned}
\Biggl(\frac{1}{\log^{1.01}P}
\sum_{\substack{|k|\leq \log^{1.01}P/2\\
|t_k|\leq 1/\log\log P}}
\mathbb{E}
\left|
F_P\left(\frac12+it_k\right)
\right|^2\Biggr)^q
&\ll
\left(\frac{1}{\log^{1.01}P}
\frac{\log^{1.01}P}{\log\log P}\log P \right)^q \\
&=\left(\frac{\log P }{\log\log P}\right)^q,
\end{aligned}
\]
which is negligible for \eqref{eq:goal-euler-product-upperbd}. We next handle the the contribution form those $k$ with $1/\log\log P\leq t_k \leq 1/2$ and the contribution form those $k$ with $-1/2\leq t_k\leq 1/\log\log P$ can be handled similarly.
By the definition of  $\mathcal{T}$, \begin{align}\label{eq:E1T}\E \textbf{1}_{k \in \mathcal{T}} |F_{P}^{\text{rand}}(1/2 + i\frac{k}{\log^{1.01}P})|^2\ll(\log^{1.1}P)^{-0.2} \E |F_{P}^{\text{rand}}(1/2 + i\frac{k}{\log^{1.01}P})|^{2.2} + (\frac{1}{\log^{1.1}P})^2.
\end{align}
Note that
\[
\prod_{ p \leq 300} \left|1 -\frac{\X(p)}{p^{1/2+ i\frac{k}{\log^{1.01}P}}}+\frac{1}{p^{1 + 2i\frac{k}{\log^{1.01}P}}}\right|^{-2.2} \ll 1
\]
 Applying Lemma \ref{lemma:mean-square-randomEulerproducts-high moment} with $\alpha=1.1,\beta=0, \sigma=0, t_1=t_2=\frac{k}{\log^{1.01}P}$ and $x=300$, we obtain \[\E |F_{P}^{\text{rand}}(1/2 + i\frac{k}{\log^{1.01}P})|^{2.2} \ll \exp\bigg(\sum_{300<p\leq P}\frac{1.21+0.11\cos(2t_1 \log p)}{p}\bigg)  \ll \log^{1.21}P(\log \log P)^{0.11}.
 \]
For the last estimate, we have used \cite[Lemma 3.5]{C11}. It follows from this and \eqref{eq:E1T} that
\begin{align*}
\E \textbf{1}_{k \in \mathcal{T}} |F_{P}^{\text{rand}}(1/2 + i\frac{k}{\log^{1.01}P})|^2&\ll(\log^{1.1}P)^{-0.2} \log^{1.21}P(\log \log P)^{0.11} + (\frac{1}{\log^{1.1}P})^2\\&\ll \log^{0.99}P(\log \log P)^{0.11}.
\end{align*}
Thus, we have
\[\Biggl( \frac{1}{\log^{1.01}P} \sum_{\substack{|k| \leq (\log^{1.01}P)/2\\1/\log\log P\leq t_k \leq 1/2}} \E \textbf{1}_{k \in \mathcal{T}} |F_{P}^{\text{rand}}(1/2 + i\frac{k}{\log^{1.01}P})|^2 \Biggr)^q \ll \log^{0.99q}P(\log \log P)^{0.11q},\]
which gives a negligible contribution to \eqref{eq:goal-euler-product-upperbd}.

Note that the contribution to \eqref{eq:goal-euler-product-upperbd} from the sum over $k \notin \mathcal{T}$ is
\begin{eqnarray}
& \leq & \sum_{\textbf{j}} \sigma^{\text{rand}}(\textbf{j}) (\E^{\textbf{j}, \text{rand}} \frac{1}{\log^{1.01}P} \sum_{\substack{|k| \leq (\log^{1.01}P)/2, \\ k \notin \mathcal{T}}} \textbf{1}_{|S_{k}(\X) - j(k)| \leq 1} |F_{P}^{\text{rand}}(1/2 + i\frac{k}{\log^{1.01}P})|^2 )^q + \nonumber \\
&& + \sum_{\textbf{j}} \sigma^{\text{rand}}(\textbf{j}) (\E^{\textbf{j}, \text{rand}} \frac{1}{\log^{1.01}P} \sum_{|k| \leq (\log^{1.01}P)/2} \textbf{1}_{|S_{k}(\X) - j(k)| > 1} \textbf{1}_{k \notin \mathcal{T}} |F_{P}^{\text{rand}}(1/2 + i\frac{k}{\log^{1.01}P})|^2 )^q . \nonumber
\end{eqnarray}
Applying H\"{o}lder's inequality to the sum over $\textbf{j}$, and using the definitions of $\E^{\textbf{j}, \text{rand}}$ and $\sigma^{\text{rand}}(\textbf{j})$, we obtain
\begin{align*}
&\ \ \ \ \ \sum_{\textbf{j}} \sigma^{\text{rand}}(\textbf{j}) (\E^{\textbf{j}, \text{rand}} \frac{1}{\log^{1.01}P} \sum_{|k| \leq (\log^{1.01}P)/2} \textbf{1}_{|S_{k}(\X) - j(k)| > 1} \textbf{1}_{k \notin \mathcal{T}} |F_{P}^{\text{rand}}(1/2 + i\frac{k}{\log^{1.01}P})|^2 )^q\\&\ll (\frac{1}{\log^{1.01}P} \sum_{|k| \leq \frac{\log^{1.01}P}{2}} \sum_{\textbf{j}} \E \prod_{i=-M}^{M} g_{j(i)}(S_{i}(\X)) \cdot \textbf{1}_{|S_{k}(\X) - j(k)| > 1} \textbf{1}_{k \notin \mathcal{T}} |F_{P}^{\text{rand}}(1/2 + i\frac{k}{\log^{1.01}P})|^2 )^q . \end{align*} 
It follows from the discussion in the final paragraph of
Section~3.5 of \cite{H23} that, for \(N\geq 1.2\log\log P\), we have
\(
\prod_{i=-M}^{M} g_{j(i)}(S_i(\X))
\mathbf{1}_{\{|S_k(\X)-j(k)|>1\}}
\mathbf{1}_{\{k\notin\mathcal{T}\}}
\leq
\delta
\prod_{\substack{-M\leq i\leq M\\ i\neq k}}
g_{j(i)}(S_i(\X)).
\)
Consequently, the contribution from the second line above is bounded by
\begin{align*}
&\ \ \ \ \ \left(
\frac{\delta}{\log^{1.01}P}
\sum_{|k|\leq \frac{\log^{1.01}P}{2}}
\sum_{-N\leq j(k)\leq N+1}
\E\left|
F_P^{\mathrm{rand}}
\left(
\frac12+\frac{ik}{\log^{1.01}P}
\right)
\right|^2
\right)^q \\
&\ll
\left(
\delta N
\sum_{\substack{n\geq 1\\  n \; \text{is} \; P \; \text{smooth}}}
\frac{1}{n}
\right)^q
\ll
\bigl(\delta N\log P\bigr)^q.
\end{align*}
This bound is acceptable for \eqref{eq:goal-euler-product-upperbd} provided that $\delta \leq \frac{1}{N\sqrt{\log\log P}}$.

To complete the proof of \eqref{eq:goal-euler-product-upperbd}
in the case \(v=0\), it remians to show that
\begin{equation}\label{eq:good-k-upperbd}
\begin{aligned}
&\ \ \ \ \sum_{\mathbf{j}}
\sigma^{\mathrm{rand}}(\mathbf{j})
\Biggl(
\E^{\mathbf{j},\mathrm{rand}}
\frac{1}{\log^{1.01}P}
\sum_{\substack{|k|\leq (\log^{1.01}P)/2\\ k\notin\mathcal{T}}}
\mathbf{1}_{|S_k(\X)-j(k)|\leq 1}
\left|
F_P^{\mathrm{rand}}
\left(
\frac12+\frac{ik}{\log^{1.01}P}
\right)
\right|^2
\Biggr)^q \\
&\ll
\Biggl(
\frac{\log P}
{1+(1-q)\sqrt{\log\log P}}
\Biggr)^q.
\end{aligned}
\end{equation}
Recall that
\(
\left|
F_P^{\mathrm{rand}}
\left(
\frac12+\frac{ik}{\log^{1.01}P}
\right)
\right|
=
\exp\Biggl(
-\Re\sum_{p\leq P}
\log\Biggl(
1-
\frac{\X(p)}
{p^{1/2+ik/\log^{1.01}P}}
+
\frac{1}
{p^{1+2ik/\log^{1.01}P}}
\Biggr)
\Biggr) 
\asymp
\exp\bigl(S_k(\X)\bigr).
\) Using $S_{k}(f) \in [j(k)-1,j(k)+1]$ and noting that  $|S_{k}(f)| \leq 1.1\log\log P + O(1)$ if $k \notin \mathcal{T}$,
it follows that left hand side of \eqref{eq:good-k-upperbd} is
\begin{align*}
& \ll  \sum_{\textbf{j}} \sigma^{\text{rand}}(\textbf{j})
\left(
\E^{\textbf{j},\text{rand}}
\frac{1}{\log^{1.01}P}
\sum_{\substack{|k|\leq(\log^{1.01}P)/2,\\
|j(k)|\leq1.1\log\log P+O(1)}}
e^{2j(k)}
\right)^q  \\
& =  \sum_{\textbf{j}} \sigma^{\text{rand}}(\textbf{j})
\left(
\frac{1}{\log^{1.01}P}
\sum_{\substack{|k|\leq(\log^{1.01}P)/2,\\
|j(k)|\leq1.1\log\log P+O(1)}}
e^{2j(k)}
\right)^q .
\end{align*}
By the definition of \(\sigma^{\text{rand}}(\textbf{j})\), the last expression can be rewritten as
\begin{align}
& \E \sum_{\textbf{j}} \prod_{i=-M}^{M} g_{j(i)}(S_{i}(\X)) (\frac{1}{\log^{1.01}P} \sum_{\substack{|k| \leq (\log^{1.01}P)/2, \\ |j(k)| \leq 1.1\log\log P + O(1)}} e^{2j(k)} )^q \nonumber \\
 \ll &\ \E \sum_{\textbf{j}} \prod_{i=-M}^{M} g_{j(i)}(S_{i}(\X)) (\frac{1}{\log^{1.01}P} \sum_{\substack{|k| \leq (\log^{1.01}P)/2, \\ |j(k)| \leq 1.1\log\log P + O(1)}} |F_{P}^{\text{rand}}(1/2 + i\frac{k}{\log^{1.01}P})|^2 )^q + \nonumber \\
& + \E \sum_{\textbf{j}} \prod_{i=-M}^{M} g_{j(i)}(S_{i}(\X)) (\frac{1}{\log^{1.01}P} \sum_{\substack{|k| \leq (\log^{1.01}P)/2, \\ |j(k)| \leq 1.1\log\log P + O(1)}} \textbf{1}_{|S_{k}(\X) - j(k)| > 1} \log^{2.2}P )^q . \nonumber
\end{align}
The second term on the right hand side is $\ll (\delta N \log^{2.2}P)^q$ (see the discussion in \cite[Section 3.6]{H23}). This gives an acceptable contribution for \eqref{eq:goal-euler-product-upperbd}  provided that $$\delta \leq \frac{1}{N \log^{1.2}P \sqrt{\log\log P}}.$$
Since $\sum_{\textbf{j}} \prod_{i=-M}^{M} g_{j(i)}(S_{i}(\X)) \equiv 1$, the first term is \[\leq \E( \frac{1}{\log^{1.01}P} \sum_{|k| \leq (\log^{1.01}P)/2} |F_{P}^{\text{rand}}(1/2 + i\frac{k}{\log^{1.01}P})|^2 )^q,\] which is $\ll \left(\frac{\log P}{1 + (1-q)\sqrt{\log\log P}}\right)^{q}$ by Lemma \ref{lemma: Multiplicative Chaos Result2}.

From the preceding discussions, we need to choose parameters \(N\) and \(\delta\) so that
\begin{align}\label{eq:condition-N-delta}
N \geq 1.2\log\log P\quad
\text{and}~\quad
\delta \leq \frac{1}{N\log^{1.2}P\sqrt{\log\log P}}.
\end{align}
In section \ref{subsecmaingorandom}, we needed to have 
\begin{align}\label{eq:condition-for-x-and-P-thm-holo1}
    x P^{32000((2M+1)/\delta)^2 \log(N\log P)} < k/100
\end{align} 
so that \eqref{eq:holomorphic-Eigenvalues-behave like-random-model1} in Proposition \ref{prop:holomorphic-Eigenvalues-behave like-random-model} could be applied. We also need  \begin{align}\label{eq:condition-for-x-and-P-thm-holo2}
\log x \leq (N\log P)^{1000(2M+1)(1/\delta)^2}
\end{align} 
to ensure \eqref{eq:contribution-of-error-terms} is acceptably small.  In section \ref{subsecpasseuler}, we need to use the fact that \begin{align}\label{eq:condition-for-x-and-P-thm-holo3}
P< x^{1/\log\log x}.
\end{align} 
In view of \eqref{eq:condition-N-delta}, we take $$N = \lceil 1.2\log\log P \rceil \quad\text{and}\quad \delta = \frac{1}{\log^{1.3}P}.$$  Recalling that  \(P\) be the
largest  number not exceeding
\(
\exp\bigl((\log L_k)^{1/6}\bigr)
\)
such that \((\log P)^{0.01}\) is an integer and $M = 2\log^{1.02}P$, we see  \eqref{eq:condition-for-x-and-P-thm-holo1}, \eqref{eq:condition-for-x-and-P-thm-holo2} and \eqref{eq:condition-for-x-and-P-thm-holo3} hold for our choice of $P$ and $ x \leq k
\exp\!\left(-(\log\log k)^2\right)$.
\qed

\begin{proof}[Proof of Theorem \ref{mainthm:low-moment-modular}] Actually, we will prove that uniformly for any $1 \leq x \leq k$ and any $0 \leq q \leq 1$, we have
\begin{equation}\label{eq:equivalent-mainthm-low-moment-modular}
\sumh_{f \in \mathcal{H}_k} |\sum_{n \leq x} \lambda_f(n)|^{2q} \ll \Biggl(\frac{x}{1 + (1-q)\sqrt{\log\log(100L_k^{'})}} \Biggr)^q , \end{equation} 
where $L_k^{'} :=\min\left\{x,\ \log(10k/x)\right\}.
$
 For $ k
\exp\!\left(-(\log\log k)^2\right) \leq x \leq k$, we have $\log\log(100L_k^{'})\asymp \log\log(100\log(10k/x))$ so that Theorem \ref{mainthm:low-moment-modular} follows from \eqref{eq:equivalent-mainthm-low-moment-modular}.

The proof of \eqref{eq:equivalent-mainthm-low-moment-modular} follows the similar argument of Theorem \ref{mainthm:low-moment-modular-sharp} with \eqref{eq:holomorphic-Eigenvalues-behave like-random-model1} in Proposition~\ref{prop:holomorphic-Eigenvalues-behave like-random-model} replaced by \eqref{eq:holomorphic-Eigenvalues-behave like-random-model2} in Proposition~\ref{prop:holomorphic-Eigenvalues-behave like-random-model}. 

Let $L_k^{''} := \min\{x^{1/\log \log x},\log(10k/x)\}$. We shall  choose $P$ to be the largest number below $L_k^{''} $ such that $\log^{0.01}P$ is an integer. Note that $\log\log L_k^{'}  \asymp \log\log L_k^{''} $, and to prove \eqref{eq:equivalent-mainthm-low-moment-modular} it will suffice to show that
\begin{equation}\label{stpmaindisplay1}
\E^{\rm holo} |\sum_{n \leq x} \lambda_f(n)|^{2q} \ll \Biggl(\frac{x}{1 + (1-q)\sqrt{\log\log P}} \Biggr)^q .
\end{equation}
In order to apply \eqref{eq:holomorphic-Eigenvalues-behave like-random-model2} in Proposition \ref{prop:holomorphic-Eigenvalues-behave like-random-model}, we need to have \begin{equation}
x P^{1000(2M+1)(1/\delta)^2 \sqrt{P}} < k/100.\end{equation} 
We also need to take   $N = \lceil 1.2\log\log P \rceil$ and $\delta = \frac{1}{\log^{1.3}P}$.
With $M = 2\log^{1.02}P$, we see this will be satisfied provided that $P^{5000 (\log^{3.62}P)\sqrt{P}} < k/100x$. We also need  $P< x^{1/\log\log x}$. This indeed holds with our choice of $P$.
\end{proof}
\begin{proof}[Proof of Theorem \ref{mainthm:low-moment-maass}]
Let $L_T^{'} := \min\{x^{1/\log \log x},T^{0.9995}/x\}$ be large. We choose $P$ to be the largest number below $\exp\{\log^{1/6}L_T^{'}\} $ such that $\log^{0.01}P$ is an integer. Since \(1\leq x\leq T^{0.999}\), we have $\log \log P \asymp \log \log x$. It suffices to show
\begin{equation}\label{eq: low-moment-maass-transform}
        \frac{12}{T^2}\sum_{j}
        \frac{\zeta(2)}{L(1,\mathrm{sym}^2 u_j)}e^{-t_j/T}\left|
        \sum_{n\leq x}\lambda_j(n)
        \right|^{2q}
        \ll
        \left(
        \frac{x}
        {1+(1-q)\sqrt{\log\log P}}
        \right)^q .
\end{equation}
  The proof of \eqref{eq: low-moment-maass-transform} is simlilar to \eqref{stpmaindisplay} in the proof of Theorem \ref{mainthm:low-moment-modular-sharp} with Lemma \ref{lemma: Eigenvalue like random model-holomorphic} and \eqref{eq:holomorphic-Eigenvalues-behave like-random-model2} in Proposition \ref{prop:holomorphic-Eigenvalues-behave like-random-model} repalecd by Lemma \ref{lemma: Eigenvalue like random model-maass} and Proposition \ref{prop:Maass Eigenvalues behave like random model} (taking $\eta=1/2000$).
\end{proof}
\begin{proof}[Proof of Corollary \ref{cor:quantitative-corollary}] We follow the proof of \cite[Corollary 1]{H23}.
Set $A:=\sqrt{\log\log(10L)}$. By Markov's inequality and
Theorem~\ref{mainthm:low-moment-modular-sharp}, for any $0\leq q\leq 1$,
\[
|\mathcal{E}|_h
\leq
\frac{1}{(\frac{\lambda \sqrt{x}}{\sqrt{A}} )^{2q}}
\sumh_{f\in\mathcal{H}_k}|S_f(x)|^{2q}
\ll
\frac{1}{\lambda^{2q}}
\left(\frac{A}{1+(1-q)A}\right)^q.
\]
Taking $q=1$, we obtain
\[
|\mathcal{E}|_h\ll \frac{A}{\lambda^2}.
\]
Next, take $q=1-\delta$, where
$\delta=(2\log\lambda)^{-1}$. Since $\lambda\geq 2$, we have
$0<\delta<1$, and hence
\[
|\mathcal{E}|_h
\ll
\frac{\lambda^{2\delta}}{\lambda^2}
\left(\frac{A}{1+\delta A}\right)^{1-\delta}
\leq
\frac{\lambda^{2\delta}}{\lambda^2\delta^{1-\delta}}
\leq
\frac{\lambda^{2\delta}}{\lambda^2\delta}
=
\frac{2e\log\lambda}{\lambda^2}.
\]
Combining these two bounds proves the corollary.
\end{proof}

\bibliographystyle{plain}

\end{document}